\documentclass[11pt,reqno]{amsart}

\usepackage{amsmath,amsfonts,amsthm,amssymb}
\usepackage{graphicx}
\usepackage{verbatim}
\usepackage{color}
\usepackage{amscd}
\usepackage{verbatim}

\begin{document}
\newcommand{\M}{{\mathcal M}}
\newcommand{\loc}{{\mathrm{loc}}}
\newcommand{\core}{C_0^{\infty}(\Omega)}
\newcommand{\sob}{W^{1,p}(\Omega)}
\newcommand{\sobloc}{W^{1,p}_{\mathrm{loc}}(\Omega)}
\newcommand{\merhav}{{\mathcal D}^{1,p}}
\newcommand{\be}{\begin{equation}}
\newcommand{\ee}{\end{equation}}
\newcommand{\mysection}[1]{\section{#1}\setcounter{equation}{0}}
\newcommand{\laplace}{\Delta}
\newcommand{\pl}{\laplace_p}
\newcommand{\grad}{\nabla}
\newcommand{\pd}{\partial}
\newcommand{\bo}{\pd}
\newcommand{\csub}{\subset \subset}
\newcommand{\sm}{\setminus}
\newcommand{\ssm}{:}
\newcommand{\diver}{\mathrm{div}\,}
\newcommand{\bea}{\begin{eqnarray}}
\newcommand{\eea}{\end{eqnarray}}
\newcommand{\bean}{\begin{eqnarray*}}
\newcommand{\eean}{\end{eqnarray*}}
\newcommand{\thkl}{\rule[-.5mm]{.3mm}{3mm}}
\newcommand{\cw}{\stackrel{\rightharpoonup}{\rightharpoonup}}
\newcommand{\id}{\operatorname{id}}
\newcommand{\supp}{\operatorname{supp}}
\newcommand{\wlim}{\mbox{ w-lim }}
\newcommand{\mymu}{{x_N^{-p_*}}}
\newcommand{\R}{{\mathbb R}}
\newcommand{\N}{{\mathbb N}}
\newcommand{\Z}{{\mathbb Z}}
\newcommand{\Q}{{\mathbb Q}}
\newcommand{\abs}[1]{\lvert#1\rvert}
\newtheorem{theorem}{Theorem}[section]
\newtheorem{corollary}[theorem]{Corollary}
\newtheorem{lemma}[theorem]{Lemma}
\newtheorem{notation}[theorem]{Notation}
\newtheorem{definition}[theorem]{Definition}
\newtheorem{remark}[theorem]{Remark}
\newtheorem{proposition}[theorem]{Proposition}
\newtheorem{assertion}[theorem]{Assertion}
\newtheorem{problem}[theorem]{Problem}
\newtheorem{conjecture}[theorem]{Conjecture}
\newtheorem{question}[theorem]{Question}
\newtheorem{example}[theorem]{Example}
\newtheorem{Thm}[theorem]{Theorem}
\newtheorem{Lem}[theorem]{Lemma}
\newtheorem{Pro}[theorem]{Proposition}
\newtheorem{Def}[theorem]{Definition}
\newtheorem{Exa}[theorem]{Example}
\newtheorem{Exs}[theorem]{Examples}
\newtheorem{Rems}[theorem]{Remarks}
\newtheorem{Rem}[theorem]{Remark}
\newtheorem{Claim}[theorem]{Claim}

\newtheorem{Cor}[theorem]{Corollary}
\newtheorem{Conj}[theorem]{Conjecture}
\newtheorem{Prob}[theorem]{Problem}
\newtheorem{Ques}[theorem]{Question}
\newtheorem*{corollary*}{Corollary}
\newtheorem*{theorem*}{Theorem}
\newcommand{\pf}{\noindent \mbox{{\bf Proof}: }}


\renewcommand{\theequation}{\thesection.\arabic{equation}}
\catcode`@=11 \@addtoreset{equation}{section} \catcode`@=12
\newcommand{\Real}{\mathbb{R}}
\newcommand{\real}{\mathbb{R}}
\newcommand{\Nat}{\mathbb{N}}
\newcommand{\ZZ}{\mathbb{Z}}
\newcommand{\CC}{\mathbb{C}}
\newcommand{\Pess}{\opname{Pess}}
\newcommand{\Proof}{\mbox{\noindent {\bf Proof} \hspace{2mm}}}
\newcommand{\mbinom}[2]{\left (\!\!{\renewcommand{\arraystretch}{0.5}
\mbox{$\begin{array}[c]{c}  #1\\ #2  \end{array}$}}\!\! \right )}
\newcommand{\brang}[1]{\langle #1 \rangle}
\newcommand{\vstrut}[1]{\rule{0mm}{#1mm}}
\newcommand{\rec}[1]{\frac{1}{#1}}
\newcommand{\set}[1]{\{#1\}}
\newcommand{\dist}[2]{$\mbox{\rm dist}\,(#1,#2)$}
\newcommand{\opname}[1]{\mbox{\rm #1}\,}
\newcommand{\mb}[1]{\;\mbox{ #1 }\;}
\newcommand{\undersym}[2]
 {{\renewcommand{\arraystretch}{0.5}  \mbox{$\begin{array}[t]{c}
 #1\\ #2  \end{array}$}}}
\newlength{\wex}  \newlength{\hex}
\newcommand{\understack}[3]{%
 \settowidth{\wex}{\mbox{$#3$}} \settoheight{\hex}{\mbox{$#1$}}
 \hspace{\wex}  \raisebox{-1.2\hex}{\makebox[-\wex][c]{$#2$}}
 \makebox[\wex][c]{$#1$}   }%
\newcommand{\smit}[1]{\mbox{\small \it #1}}
\newcommand{\lgit}[1]{\mbox{\large \it #1}}
\newcommand{\scts}[1]{\scriptstyle #1}
\newcommand{\scss}[1]{\scriptscriptstyle #1}
\newcommand{\txts}[1]{\textstyle #1}
\newcommand{\dsps}[1]{\displaystyle #1}
\newcommand{\dx}{\,\mathrm{d}x}
\newcommand{\dy}{\,\mathrm{d}y}
\newcommand{\dz}{\,\mathrm{d}z}
\newcommand{\dt}{\,\mathrm{d}t}
\newcommand{\dr}{\,\mathrm{d}r}
\newcommand{\du}{\,\mathrm{d}u}
\newcommand{\dv}{\,\mathrm{d}v}
\newcommand{\dV}{\,\mathrm{d}V}
\newcommand{\ds}{\,\mathrm{d}s}
\newcommand{\dS}{\,\mathrm{d}S}
\newcommand{\dk}{\,\mathrm{d}k}

\newcommand{\dphi}{\,\mathrm{d}\phi}
\newcommand{\dtau}{\,\mathrm{d}\tau}
\newcommand{\dxi}{\,\mathrm{d}\xi}
\newcommand{\deta}{\,\mathrm{d}\eta}
\newcommand{\dsigma}{\,\mathrm{d}\sigma}
\newcommand{\dtheta}{\,\mathrm{d}\theta}
\newcommand{\dnu}{\,\mathrm{d}\nu}

\def\ga{\alpha}     \def\gb{\beta}       \def\gg{\gamma}
\def\gc{\chi}       \def\gd{\delta}      \def\ge{\epsilon}
\def\gth{\theta}                         \def\vge{\varepsilon}
\def\gf{\phi}       \def\vgf{\varphi}    \def\gh{\eta}
\def\gi{\iota}      \def\gk{\kappa}      \def\gl{\lambda}
\def\gm{\mu}        \def\gn{\nu}         \def\gp{\pi}
\def\vgp{\varpi}    \def\gr{\rho}        \def\vgr{\varrho}
\def\gs{\sigma}     \def\vgs{\varsigma}  \def\gt{\tau}
\def\gu{\upsilon}   \def\gv{\vartheta}   \def\gw{\omega}
\def\gx{\xi}        \def\gy{\psi}        \def\gz{\zeta}
\def\Gg{\Gamma}     \def\Gd{\Delta}      \def\Gf{\Phi}
\def\Gth{\Theta}
\def\Gl{\Lambda}    \def\Gs{\Sigma}      \def\Gp{\Pi}
\def\Gw{\Omega}     \def\Gx{\Xi}         \def\Gy{\Psi}

\renewcommand{\div}{\mathrm{div}}
\newcommand{\red}[1]{{\color{red} #1}}


\title{A spectral result for Hardy inequalities}

\author{Baptiste Devyver}
\address{Baptiste Devyver, Department of Mathematics,  Technion - Israel Institute of Technology, Haifa 32000, Israel}
\email{baptiste.devyver@univ-nantes.fr \\
devyver@tx.technion.ac.il}

\begin{abstract}

Let $P$ be a linear, elliptic second order symmetric operator, with an associated quadratic form $q$, and let $W$ be a potential such that the Hardy inequality 

$$\lambda_0\int_\Omega Wu^2\leq q(u)$$
holds with (non-negative) best constant $\lambda_0$. We give sufficient conditions so that the spectrum of the operator $\frac{1}{W}P$ is $[\lambda_0,\infty)$. In particular, we apply this to several well-known Hardy inequalities: (improved) Hardy inequalities on a bounded convex domain of $\R^n$ with potentials involving the distance to the boundary, and Hardy inequalities for minimal submanifolds of $\R^n$.

\end{abstract}

\maketitle

\section{Introduction}

Let $P$ be a linear elliptic, second order, symmetric, non-negative operator on a domain $\Omega$, and let $q$ be the quadratic form associated to $P$. Following Carron \cite{Carron} and Tertikas \cite{Tertikas}, we will call \textit{Hardy inequality} for $P$ with weight $W\geq 0$ and constant $\lambda>0$, the following inequality:

\begin{equation}\label{hardy}
\lambda \int_\Omega Wu^2\leq q(u),\,\,\forall u\in C_0^\infty(\Omega).
\end{equation}
We denote by $\lambda_0=\lambda_0(\Omega,P,W)$ the best constant $\lambda$ for which inequality \eqref{hardy} is valid. By convention, if \eqref{hardy} does not hold for any $\lambda>0$, we will let $\lambda_0=0$. The inequality \eqref{hardy} aims to quantify the positivity of $P$: for instance, inequality \eqref{hardy} with $W\equiv1$ is equivalent to the positivity of the bottom of the spectrum of (the Friedrichs extension of) $P$. 

Let us give a celebrated example of Hardy inequality for $P=-\Delta$, which will be a guideline for us in this paper (see \cite{MMP} for the convex case, and \cite{BFT2}, \cite{Psaradakis} for the mean convex case):

\begin{Exa}\label{ex_convex}

{\em If $\Omega$ is a $C^2$, bounded, mean convex domain of $\R^n$, and $\delta$ is the distance to the boundary of $\Omega$, then

\begin{equation}\label{convex}
\frac{1}{4}\int_\Omega\frac{u^2}{\delta^2}\leq \int_\Omega |\nabla u|^2,\,\,\forall u\in C_0^\infty(\Omega).
\end{equation}
}
\end{Exa}
We recall that $\Omega$ is called \textit{mean convex} if the mean curvature of its boundary is non-negative.\\

Let us return to the general case. 

\subsection{The best constant and the existence of minimizers}
A natural question is, for a given weight $W\geq0$ such that the Hardy inequality \eqref{hardy} holds, to compute the best constant $\lambda_0$ and to discuss whether $\lambda_0$ is attained by a minimizer in the appropriate space or not. More precisely, define $\mathcal{D}^{1,2}$ to be the completion of $C_0^\infty(\Omega)$ with respect to the norm $\sqrt{q}$. The variationnal problem associated to the Hardy inequality \eqref{hardy} is 

\begin{equation}\label{variational}
\lambda_0=\inf_{u\in \mathcal{D}^{1,2}\setminus\{0\}}\frac{q(u)}{\int_\Omega Wu^2}.
\end{equation}
If \eqref{variational} is not realized by a function in $\mathcal{D}^{1,2}$, we will say that the Hardy inequality \eqref{hardy} with best constant $\lambda_0$ does not have a minimizer.

Another interesting quantity, related to the existence of minimizers, is the \textit{best constant at infinity}. It is defined as follows (see \cite{Ag2}, \cite{MMP} and \cite{DFP}):

\begin{Def}
{\em 
The {\em best constant at infinity} $\lambda_\infty=\lambda_\infty(\Omega,P,W)$ is the supremum of the set of $\alpha\geq0$ such that

$$\alpha\int_{\Omega\setminus K_\alpha} Wu^2\leq q(u),\,\,\forall u\in C_0^\infty(\Omega\setminus K_\alpha),$$
for some $K_\alpha\subset\subset\Omega$ compact subset of $\Omega$.
}
\end{Def}
Closely related notions have been introduced (under different names) in \cite{Tertikas} and \cite{FT}.

\begin{Exa}[Example \ref{ex_convex}, continued]
{\em

For the Hardy inequality \eqref{convex} with $W=\frac{1}{\delta^2}$, the best constant and the best constant at infinity are both equal to $\frac{1}{4}$. Furthermore, inequality \eqref{convex} does not have a minimizer. 

}
\end{Exa}

\subsection{Improving Hardy inequalities}
A very natural question is the following: for the Hardy inequality \eqref{hardy} with the best constant $\lambda_0$, can we improve it by adding another non-negative potential $V$ to the left-hand side, i.e. is there a positive constant $\mu$ and a non-negative, non-zero potential $V$ such that 

\begin{equation}\label{improved_hardy}
\lambda_0 \int_\Omega Wu^2+\mu\int_\Omega Vu^2\leq q(u),\,\,\forall u\in C_0^\infty(\Omega) 
\end{equation}
Of course, such an improvement -- if it exists -- is not at all unique, and finding a potential $V$ which is ``as large as possible'' is important. \\

Results by Filippas-Tertikas \cite{FT}, Agmon \cite{Ag2}, Pinchover \cite{P}, \cite{P3}, Marcus-Mizel-Pinchover \cite{MMP}, Pinchover-Tintarev \cite{PT} among others show that the best constant at infinity, as well as the existence of minimizers, play an important role in this problem: indeed, a general result obtained by Agmon \cite{Ag2} (see also Pinchover \cite{P3}, Lemma 4.6 for an easier and more general proof) shows that if the best constant at infinity is strictly greater than the best constant in the inequality \eqref{hardy}, then no improvement by a non-negative, non-zero potential $V$ is possible. 

Also, concerning the minimizers, Pinchover-Tintarev show (Lemma 1.1 in \cite{PT}) that if $W>0$ and if there is a minimizer for the Hardy inequality \eqref{hardy},
then no improvement by a non-negative, non-zero potential $V$ is possible. In fact, if a minimizer exists then it is a \textit{ground state} (in the sense of Agmon) of $P-W$. \\

The possibility of adding a potential $V\gneqq0$ in the left-hand side of the Hardy inequality \eqref{hardy} with best constant $\lambda_0$ has to do with the \textit{criticality} of the operator $P-\lambda_0W$: there exists $V\gneqq0$ such that the improved inequality \eqref{improved_hardy} is valid (with $\mu=1$) if and only if $P-\lambda_0W$ is \textit{subcritical} in $\Omega$. See for example \cite{DFP} for details and references on this.\\

In the case that the best constant at infinity is strictly larger than the best constant, or that there is a minimizer, the potential $W$ ``does not grow fast enough at infinity'' in a certain sense. Let us illustrate this by the following example, related to Example \ref{ex_convex}:

\begin{Exa}

{\em 

Let $P=-\Delta$ on a smooth, bounded, mean convex domain $\Omega$, and let us define, for $\alpha\in \R$,

$$W_\alpha=\delta^{-\alpha},$$
where we recall that $\delta$ is the distance to the boundary of $\Omega$. Then, 

$$\lambda_\infty(\Omega,-\Delta,W_\alpha)=\left\{\begin{array}{lcl}
+\infty,\,\alpha\in (-\infty,2)&&\\
\frac{1}{4},\,\alpha=2&&\\
0,\,\alpha \in (2,+\infty)&&
\end{array}\right.$$

}

\end{Exa}
More generally, if $W$ is a \textit{small perturbation} of $P$ (see \cite{P2}), then $\lambda_\infty(\Omega,P,W)=\infty$.

\begin{Def}

{\em

A non-negative potential $W$, satisfying the Hardy inequality \eqref{hardy}, is said to belong to the class of {\em admissible} potentials $\mathcal{A}(\Omega,P)$ if $$\lambda_0(\Omega,P,W)=\lambda_\infty(\Omega,P,W),$$ 
and if there is no minimizer of the associated variational problem \eqref{variational}. When the dependence with respect to $\Omega$ and $P$ will be clear, we will write $\mathcal{A}$ instead of $\mathcal{A}(\Omega,P)$.

}

\end{Def}
For instance, it follows from the results of Marcus-Mizel-Pinchover \cite{MMP} that if $\Omega$ is a $C^2$, bounded, mean convex domain, then the potential $\delta^{-2}$ belongs to the class $\mathcal{A}(\Omega,-\Delta)$. In this article, we will focus on Hardy inequalities \eqref{hardy} with admissible potentials. A Hardy inequality \eqref{hardy} with admissible potential can sometimes be improved, but \textit{not always}. Actually, if $W$ is admissible, deciding whether the Hardy inequality \eqref{hardy} with best constant $\lambda_0$ can be improved or not, is a delicate question. We define a subclass of $\mathcal{A}$:

\begin{Def}

{\em 

A non-negative potential $W$, satisfying the Hardy inequality \eqref{hardy}, is said to belong to the class of {\em optimal} potentials $\mathcal{O}(\Omega,P)$ if $W\in \mathcal{A}(\Omega,P)$ and if the Hardy inequality \eqref{hardy} with best constant $\lambda_0$ cannot be improved, i.e. if there is no $V\gneqq0$, $\mu>0$ such that inequality \eqref{improved_hardy} holds. Equivalently, $W\in \mathcal{O}(\Omega,P)$ if and only if $W\in \mathcal{A}(\Omega,P)$ and $P-\lambda_0W$ is {\em critical} (see \cite{PT2}).  Furthermore, $P-\lambda_0 W$ is critical if and only if it does not have a ground state in the sense of Agmon (see \cite{PT2}); this of course includes an obvious case when such a ground state is a true minimizer).

}

\end{Def}

\begin{Exa}\label{Hardy_cla}
{\em

The potential $\frac{1}{|x|^2}$ is an optimal potential for $P=-\Delta$ on $\R^n$ (or equivalently, on $\R^n\setminus\{0\}$) for $n\geq3$ (see \cite{DFP} for a short proof). Indeed, the operator $-\Delta-\left(\frac{n-2}{2}\right)^2\frac{1}{|x|^2}$ is critical and has ground state $|x|^{\frac{2-n}{2}}$. Recall that the potential $\frac{1}{|x|^2}$ appears in the classical Hardy inequality in $\R^n$, $n\geq3$ with best constant:

$$\left(\frac{n-2}{2}\right)^2\int_{\R^n}\frac{u^2(x)}{|x|^2}\leq \int_{\R^n}|\nabla u|^2,\,\,\forall u\in C_0^\infty(\R^n).$$
}

\end{Exa}

However, the Hardy inequality \eqref{convex} of Example \ref{ex_convex} can be improved. The first improvement of inequality \eqref{convex} was obtained by Brezis and Vazquez \cite{BV}, for $V\equiv\mathbf1$, however $\mathbf1\notin \mathcal{A}(\Omega,-\Delta-\frac{1}{4\delta^2})$ -- in fact, $V\equiv\mathbf1$ is a small perturbation of $-\Delta-\frac{1}{4\delta^2}$, and thus $\lambda_\infty(\Omega,-\Delta-\frac{1}{4\delta^2},\mathbf1)=\infty$. Later, an improvement by a potential in the class $\mathcal{A}(\Omega,-\Delta-\frac{1}{4\delta^2})$ was obtained by Brezis and Marcus \cite{BM}. Let us introduce the normalized logarithm function, defined by

$$X_1(t):=\left(1-\log t\right)^{-1}.$$
A consequence of the work of Brezis and Marcus is the following:

\begin{Exa}\label{ex_improved_convex}

{\em Let $\Omega$ be a smooth, bounded, mean convex domain, then we have the improved Hardy inequality

\begin{equation}\label{improved_convex}
\frac{1}{4}\int_\Omega \frac{u^2}{\delta^2}X^2_1\left(\frac{\delta}{D}\right)\leq \int_\Omega |\nabla u|^2-\frac{1}{4}\int_\Omega\frac{u^2}{\delta^2},\,\,\forall u\in C_0^\infty(\Omega),
\end{equation}
where $D$ is any constant such that

$$D\geq \sup_{x\in \Omega} \delta(x).$$
Furthermore, $\frac{1}{4}$ is the best constant and the best constant at infinity, and there is no minimizer. In particular, $\delta^{-2}X^2_1 \left(\frac{\delta}{D}\right)$ belongs to the class $\mathcal{A}(\Omega,-\Delta-\frac{1}{4\delta^2})$.

}

\end{Exa}
More recently, Barbatis, Filippas and Tertikas \cite{BFT} have obtained a series of successive improvements of the Hardy inequality \eqref{convex} with admissible potentials, generalizing the improved Hardy inequality \eqref{improved_convex} obtained by Brezis and Marcus. In order to present their result, let us define for $i\geq1$,

$$X_i(t):=X_1(X_{i-1}(t)),$$
and by convention $X_0\equiv\mathbf1$. Let us also define

$$\mathcal{W}_i:=\frac{1}{4\delta^2}\left(\sum_{k=0}^{i} X_0^2\left(\frac{\delta}{D}\right)\cdots X_k^2\left(\frac{\delta}{D}\right)\right),$$
and

$$\mathcal{J}_i:=\mathcal{W}_i-\mathcal{W}_{i-1}=\frac{1}{4\delta^2}X_0^2\left(\frac{\delta}{D}\right)\cdots X_i^2\left(\frac{\delta}{D}\right).$$
The result of Barbatis, Filippas and Tertikas obtained in \cite{BFT} is presented in the following example:

\begin{Exa}\label{ex_series_convex}

{\em Let $\Omega$ be a $C^2$, bounded, mean convex domain, then we can choose $D> \sup_{x\in \Omega} \delta(x)$ big enough (see Remark \ref{indep of i}) so that for every $i\geq1$, the following improved Hardy inequality holds

\begin{equation}\label{improved_H}
\int_\Omega \mathcal{J}_{i} u^2\leq \int_\Omega |\nabla u|^2-\int_\Omega \mathcal{W}_{i-1} u^2,\,\forall u\in C_0^\infty(\Omega).
\end{equation}
Furthermore, $1$ is the best constant and the best constant at infinity, and there is no minimizer. In particular, $\mathcal{J}_{i}$ belongs to the class $\mathcal{A}(\Omega,-\Delta-\mathcal{W}_{i-1}).$

}

\end{Exa}

\subsection{The spectrum of $\frac{1}{W}P$} In this subsection, we collect some results concerning  the spectrum of operators of the type $\frac{1}{W}P$. For more details on this, see \cite{Ag2}, Section 3 in \cite{MMP}, or Proposition 4.2 in \cite{DFP}.\\

If the potential $W$ in the Hardy inequality \eqref{hardy} is \textit{positive}, then there is a spectral interpretation of the best constant $\lambda_0(\Omega,P,W)$ and of the best constant at infinity $\lambda_\infty(\Omega,P,W)$. Consider (the Friedrichs extension of)  the operator $\frac{1}{W}P$, which is self-adjoint on $L^2(\Omega,W\mathrm{d}x)$. Then $\lambda_0(\Omega,P,W)$, the best constant in \eqref{hardy} is the infimum of the spectrum of $\frac{1}{W}P$, and $\lambda_\infty(\Omega,P,W)$, the best constant at infinity is the infimum of the essential spectrum of $\frac{1}{W}P$ (the result concerning $\lambda_\infty$ comes from the Persson's formula, see \cite{Persson}, \cite{Ag3}). As an immediate consequence, the following result holds:

\begin{Lem}\label{discrete}

{\em

Assume that $\lambda_\infty(\Omega,P,W)=+\infty$. Then the spectrum of $\frac{1}{W}P$ is discrete. 

}

\end{Lem}

\begin{Exa}\label{discrete_ex}

{\em

Let $\Omega$ be a smooth, bounded domain. Then for every $\alpha\in(2,\infty)$, the spectrum of $\delta^{\alpha}$ is discrete. Indeed, for every $\alpha\in (2,\infty)$, $\lambda_\infty(\Omega,-\Delta,\delta^{-\alpha})=\infty$.

}

\end{Exa}
Also, there is a minimizer of the variational problem \eqref{variational} if and only if $\lambda_0(\Omega,P,W)$ is an eigenvalue of $\frac{1}{W}P$. In \cite{MMP}, p.3246, the authors attributed the following result to Agmon: 

\begin{Claim}[Agmon]\label{Agmon} On a smooth, bounded, mean convex domain, the spectrum of $\delta^2(-\Delta)$ is $[\frac{1}{4},\infty)$. Furthermore, without the mean convexity assumption on the domain, the essential spectrum of $\delta^2(-\Delta)$ is $[\frac{1}{4},\infty)$.

\end{Claim}
To the author's knowledge, Agmon never published this result, but proved a closely related result in \cite{Agmon}. The validity of Agmon's claim \ref{Agmon} implies at once that if $\Omega$ is a smooth, bounded, mean convex domain, the best constant and the best constant at infinity for the Hardy inequality \eqref{ex_convex} are both equal to $\frac{1}{4}$. Agmon's claim has to be compared to Example \ref{discrete_ex}.

\subsection{Optimal Hardy inequalities and the supersolution construction}
In the article \cite{DFP}, starting from a general subcritical operator $P$ on a punctured domain $\Omega^\star=\Omega\setminus\{0\}$, we have constructed an optimal potential $W$ satisfying the Hardy inequality \eqref{hardy} with best constant $1$. This is actually a generalization of Example \ref{Hardy_cla}, since for the case $P=-\Delta$ with $\Omega=\R^n$, $n\geq3$, the constructed potential $W$ is equal to $\left(\frac{n-2}{2}\right)^2\frac{1}{|x|^2}$. The most remarkable property of these constructed potentials $W$ is the criticality of $P-\lambda_0W$: as we have mentionned above, if $W\in \mathcal{A}(\Omega,P)$, the criticality of $P-\lambda_0 W$ is a delicate property. 

These optimal potentials are obtained through a construction that we have called in \cite{DFP} the \textit{supersolution construction}, and that we recall now: 

\begin{Pro}[Supersolution construction, see \cite{DFP}, Lemma 5.1 and Corollary 5.2]

Assume that $P$ is subcritical in $\Omega$, and that there exist $u_0$ and $u_1$ two linearly independent, positive supersolutions of $P$. Then the non-negative potential

$$W(u_0,u_1):=\frac{1}{4}\left|\nabla \log\left(\frac{u_0}{u_1}\right)\right|^2$$
satisfies the Hardy inequality \eqref{hardy} with $\lambda=1$.

\end{Pro}
In fact, the above norm $|\cdot|$ is a norm associated with $P$, as we shall explain later.
Moreover, the results of \cite{DFP} show that if $u_0$ and $u_1$ are positive \textit{solutions} of $Pu=0,$ in $\Omega\setminus\{0\}$, and if moreover

$$\lim_{x\to\infty}\frac{u_0(x)}{u_1(x)}=0,\,\,\,\,\lim_{x\to0}\frac{u_0(x)}{u_1(x)}=\infty,$$
then $W(u_0,u_1)$ is an optimal potential (having a singularity at $0$). Actually, $u_0$ as above is necessarily the minimal, positive Green function of $P$ with pole $0$. These optimal potentials obtained \textit{via} the supersolution construction have another interesting property: if $W(u_0,u_1)$ is \textit{positive}, then the spectrum of the operator $\frac{1}{W(u_0,u_1)}P$ is $[1,\infty)$. The proof of this last fact relies on the existence of generalized eigenfunctions of exponential type for $\frac{1}{W}P-\lambda$, for every $\lambda\geq1$.\\

However, the optimal potentials $W(u_0,u_1)$ have two drawbacks: first, the existence of two \textit{linearly independent} positive solutions $u_0$, $u_1$ of 

$$Pu_i=0$$
does not always hold: one generally needs to remove one point in $\Omega$ in order to guarantee this condition, and the potential obtained will have a singularity at that point. Secondly, finding the asymptotic of $W(u_0,u_1)$ at infinity, or even a lower bound of $W(u_0,u_1)$ in order to get a more explicit Hardy inequality, is a difficult problem. Actually, it is an open problem to give sufficient conditions guaranteeing that $W(u_0,u_1)$ is positive in a neighborhood of infinity. This is an interesting question, since the interpretation of the best constant and the best constant at infinity in term of the spectrum of the operator $\frac{1}{W(u_0,u_1)}P$ holds only if $W(u_0,u_1)$ is positive.\\

It is sometimes more natural to work with non-optimal (but good enough) potentials given by an explicit formula. In the supersolution construction, instead of two positive solutions $u_i$ of $Pu=0,$ one can take two (well chosen) positive supersolutions. For example, the  potential $\frac{1}{4\delta^2}$ of the Hardy inequality of Example \ref{ex_convex} is obtained by applying the supersolution construction with $u_0=1$ and $u_1=\delta$ (which is a supersolution, but not a solution of $P=-\Delta$). Furthermore, the potentials of the Examples \ref{ex_convex}, \ref{ex_improved_convex} and \ref{ex_series_convex} are not optimal, because the corresponding Hardy inequality can be improved, yet they are admissible potentials. We will also see (in Proposition \ref{Fi-Ter}) that each improvement of the Hardy inequality of Example \ref{ex_series_convex} can also be obtained using the supersolution construction, with explicit supersolutions $u_0$ and $u_1$ which are functions of $\delta$. It is thus tempting to ask whether the property that the spectrum of $\frac{1}{W(u_0,u_1)}P$ is $[1,\infty)$, valid in the case of the \textit{optimal} potentials obtained in \cite{DFP}, remains true for only \textit{admissible} potentials (that is, potentials in the class $\mathcal{A}(\Omega,P)$. In particular, it is interesting to ask the following question:\\

\begin{Ques} \label{spec}
{\em
For the Hardy inequalities of Examples \ref{ex_convex} and \ref{ex_series_convex} obtained respectively by Brezis-Marcus and Barbatis-Filippas-Tertikas, does it hold that the spectrum of $\frac{1}{W(u_0,u_1)}P$ is $[1,\infty)$?
}
\end{Ques}

\subsection{Our results and organization of the article}
In this article, we generalize the results of \cite{DFP} concerning the spectrum of $\frac{1}{W(u_0,u_1)}P$ for optimal potentials $W(u_0,u_1)$, to non-optimal potentials $W(u_0,u_1)$ obtained by the supersolution construction, provided that the function $u_0$ and $u_1$ are \textit{optimal approximate solutions} of $P$ at infinity. We will define later what \textit{optimal approximate solutions} of $P$ at infinity precisely means (see Definition \ref{approx_sol}), since it is technical. One of the main results of this article can then be roughly stated as follows (for a precise formulation, see Theorem \ref{general_spec}):

\begin{Thm}\label{general}

Let $(u_0,u_1)$ be a pair of optimal approximate solutions at infinity for $P$. Recall that $W(u_0,u_1):=\frac{1}{4}\left|\nabla \log \frac{u_0}{u_1}\right|^2$. Then the essential spectrum of $\frac{1}{W(u_0,u_1)}P$ is equal to $[1,\infty)$, and its spectrum below $1$ consists at most of a finite number of eigenvalues with finite multiplicity.

\end{Thm}
Roughly speaking, the proof relies on the fact that there are \textit{approximate} generalized eigenfunctions of exponential type for $\frac{1}{W(u_0,u_1)}P$. As a main corollary of Theorem \ref{general}, we will be able to answer Question \ref{spec} (see Theorem \ref{improved_hardy}):

\begin{Thm}\label{convex_domain_spec}

Let $\Omega$ be a $C^2$, bounded, mean convex domain of $\R^n$, and assume that $D$ is chosen as in Example \ref{ex_series_convex}. Then for every $i\geq1$, the spectrum of $\frac{1}{\mathcal{J}_i}\left(-\Delta-\mathcal{W}_{i-1}\right)$ is $[1,\infty)$. Furthermore, without the mean convexity assumption on $\Omega$, the {\em essential} spectrum of $\frac{1}{\mathcal{J}_i}\left(-\Delta-\mathcal{W}_{i-1}\right)$ is $[1,\infty)$, and the intersection of the spectrum of $\frac{1}{\mathcal{J}_i}\left(-\Delta-\mathcal{W}_{i-1}\right)$ with $(-\infty,1)$ consists of (at most) a finite number of eigenvalues with finite multiplicity.

\end{Thm}
In particular, this proves Agmon's claim \ref{Agmon} about the spectrum of $\frac{1}{\delta^2}(-\Delta)$. As an immediate corollary, we recover the value of the best constant and of the best constant at infinity in the improved Hardy inequalities \eqref{improved_H}, a result already proved in \cite{BFT}. \\

Another result that we obtain as a consequence of Theorem \ref{general} concerns Hardy inequalities on minimal submanifolds $M^n\hookrightarrow \R^N$. Let us fix $x_0\in \R^N$, and denote by $r:=d_{\R^N}(x_0,\cdot)$, where $d_{\R^N}$ is the Euclidean distance. Carron obtained in \cite{Carron} the following Hardy inequality

\begin{equation}\label{hardy_mini}
\left(\frac{n-2}{2}\right)^2\int_M\frac{u^2}{r^2}\leq \int_M|\nabla u|^2,\,\forall u\in C_0^\infty(M).
\end{equation}
Let us denote by $\mathrm{II}$ the second fundamental form of the isometric immersion $M^n\hookrightarrow \R^N$, and denote by $\Delta_M$ the (negative) Laplacian on $M$. As a consequence of Theorem \ref{general}, we get the following result (see Theorem \ref{optimal_mini}):

\begin{Thm}\label{minimal_spec}
Assume that the total curvature of the minimal isometric immersion $M^n\hookrightarrow \R^N$ is finite, i.e. that

$$\int_M |\mathrm{II}|^{\frac{n}{2}}<\infty.$$
Then the spectrum of $r^2(-\Delta_M)$ is $\big[\left(\frac{n-2}{2}\right)^2,\infty\big)$.

\end{Thm}
To conclude this introduction, we propose the following open problem:

\begin{Ques} 
{\em 
Are there embedded eigenvalues for the operators $\frac{1}{\mathcal{J}_i}\left(-\Delta-\mathcal{W}_{i-1}\right)$ and $r^2(-\Delta_M)$, appearing respectively in Theorem \ref{convex_domain_spec} and Theorem \ref{minimal_spec}?
}
\end{Ques}

The structure of the article is as follows: in Section \ref{Sec:general_theory}, we establish the general result (Theorem \ref{general}) that will be the key to studying the spectrum of our operators. In Section \ref{Sec:bounded_domain}, we apply this result to the Hardy inequality \eqref{improved_H}, and we prove Theorem \ref{convex_domain_spec}. In Section \ref{Sec:minimal_immersions}, we consider the case of the Hardy inequality on minimal submanifolds of $\R^n$, and we prove Theorem \ref{minimal_spec}. In Section \ref{Sec:Agmon_metric}, we study an exponential volume growth property, which shows up naturally for Hardy inequalities with admissible potentials.\\

\section{General theory}\label{Sec:general_theory}

\subsection{Preliminaries}
{\em Throughout the paper, the potentials appearing in the various Hardy inequalities will always be non-negative. As usual, $C$ will denote a generic constant, whose value can change from line to line.}\\

\textit{Notation:} For two positive functions $f$ and $g$, we will write $f\asymp g$ if there is a positive constant $C$ such that

$$C^{-1}f\leq g\leq Cf.$$
\vskip5mm
Let $\Omega$, $n\geq 2$ be a smooth domain in $\R^n$ (or more generally, a smooth, connected manifold of dimension $n$). The \textit{infinity} of $\Omega$ is the ideal point in the one-point compactification of $\Omega$. From this, we derive the notions of neighborhood of infinity in $\Omega$, of convergence at infinity for a real function defined in $\Omega$, etc... Let $\nu$ be a positive measure on $\Omega$. Consider a symmetric second-order elliptic operator $L$ on $\Omega$ with real coefficients in divergence form of the type

\begin{equation} \label{div_P}
Pu=- \div \big(A \grad u\big) +cu,
\end{equation}
Here, $-\div$ is the formal adjoint of the gradient with respect to the measure $\nu$.  We assume that for every $x\in\Omega$ the matrix $A(x)=\big(a^{ij}(x)\big)_{i,j}$ is symmetric and that the real quadratic form
\be\label{ellip}
 \langle\xi,A(x) \xi\rangle := \sum_{i,j =1}^n \xi_i a^{ij}(x) \xi_j \qquad
 \xi \in \Real ^n
\end{equation}
is positive definite. We will denote the norm associated to this quadratic form by $|\cdot|_A$, that is 

$$|\xi|_A^2:=\langle\xi,A(x) \xi\rangle.$$
Moreover, it is assumed that $P$ is locally uniformly elliptic, and that $A$ and $c$ are locally bounded in $\Omega$.

We denote by $q$ the quadratic form associated to $P$, defined by

$$q(u)=\int_\Omega\left( \langle A\nabla u,\nabla u\rangle+cu^2\right)\mathrm{d}\nu ,\,\,\forall u\in C_0^\infty(\Omega).$$
We will assume that $q$ is nonnegative, and consider the Friedrichs extension of $P$, that we will also denote $P$. It is a self-adjoint operator on $L^2(\Omega,\mathrm{d}\nu)$. If $u\in W^{1,2}_{loc}(\Omega)$ and $f\in L_{loc}^\infty(\Omega)$, we will say that the equation

$$Pu=f$$
holds in $\Omega$ in the \textit{weak sense} if for every $\varphi\in C_0^\infty(\Omega)$,

$$\int_\Omega \left( \langle A\nabla u,\nabla \varphi\rangle+cu\varphi\right)\mathrm{d}\nu=\int_\Omega f\varphi \,\mathrm{d}\nu.$$
In a similar way, we define the notion of weak supersolutions/subsolutions for $P$.\\

We will make use of the following two constructions. The first one is called the \textit{h-transform} and is defined as follows: if $h>0$ is in $C^{1,\alpha}_{loc}(\Omega)$ and $Ph\in L_{loc}^\infty(\Omega)$, we define the operator 

$$P_h:=h^{-1}Ph,$$
which is self-adjoint on $L^2(\Omega,h^2\mathrm{d}\nu)$. Also, the following formula holds in the weak sense:

\begin{equation}\label{h-transform}
P_hu=-\div_{h}(A\nabla u)+\frac{Ph}{h}u,
\end{equation}
where 

$$\div_{h}(X)=\div(X)+2\langle h^{-1}\nabla h,X\rangle$$
is the divergence with respect to the measure $h^2\mathrm{d}\nu$.\\

We call the second construction the \textit{change of measure}. Let $W>0$, then we can consider the operator $\frac{1}{W}P$. It is a self-adjoint operator on $L^2(\Omega,W\mathrm{d}\nu)$, and its quadratic form is the same as $P$. If $\div_W$ is the divergence with respect to the measure $W\mathrm{d}\nu$, then 

\begin{equation}\label{change_measure}
\frac{1}{W}Pu=-\div_W\left(\frac{A}{W} \grad u \right)+\frac{c}{W}u,
\end{equation}
which follows from the formula

$$\div_W(X)=\frac{1}{W}\div(WX).$$

\subsection{A spectral result}

We now present a general spectral result concerning operators of the form $\frac{1}{W}P$. It is probable that this result, maybe in a weaker form, is already known to experts, even if we have been unable to find a reference in the literature that covers such a general case.

\begin{Thm}\label{ess_spectrum_equiv_pot}
Assume that $W_1$ and $W_2$ are two positive (in a neighborhood of infinity) potentials in $L^\infty_{loc}(\Omega)$, such that

$$W_1(x)\sim W_2(x)\mbox{ when }x\to\infty.$$
Then $\sigma_{ess}(\frac{1}{W_1}P)=\sigma_{ess}(\frac{1}{W_2}P)$.

\end{Thm}

\begin{proof}

Let $\lambda\in \sigma_{ess}(\frac{1}{W_1}P)$. Let $(\Omega_n)_{n\in\mathbb{N}}$ be an exhaustion of $\Omega$. Then, using a ground state transform, it follows essentially from the decomposition principle for the essential spectrum (see \cite{DL} or \cite{Gl}) that for every compact set $K\Subset\Omega$, there is a Weyl sequence $(u_n)_{n\in\mathbb{N}}$ associated to $\lambda$, orthonogal in $L^2(W_1\mathrm{d}\nu)$, such that for every $n\in \mathbb{N}$, the support of $u_n$ is included in $\Omega\setminus \Omega_n$. By definition of a Weyl sequence, if we denote by $||\cdot||$ the norm in $L^2(W_1\mathrm{d}\nu)$, there holds

$$\lim_{n\to\infty}\frac{||(\frac{1}{W_1}P-\lambda)u_n||}{||u_n||}=0.$$
Since $W_1(x)\sim_{x\to\infty}W_2(x)$ and given the hypothesis on the support of $u_n$, one has, when $n\to\infty$,

$$||u_n||\sim ||u_n||_{L^2(W_2\mathrm{d}\nu)}.$$
Also, using the hypothesis on the support of $u_n$, and the fact that $W_1$ and $W_2$ are equivalent at infinity, one has for $n$ big enough,

$$\begin{array}{rcl}
||(\frac{1}{W_2}P-\lambda)u_n||_{L^2(W_2\mathrm{d}\nu)}&=&\left(\int_\Omega |(P-\lambda W_2)u_n|^2\frac{\mathrm{d}\nu}{W_2}\right)^{1/2}\\\\
&\leq& \left(\int_\Omega |(P-\lambda W_1)u_n|^2\frac{\mathrm{d}\nu}{W_2}\right)^{1/2}\\\\
&&+|\lambda|\left(\int_\Omega |(W_1- W_2)u_n|^2\frac{\mathrm{d}\nu}{W_2}\right)^{1/2}\\\\
&\leq& 2||(\frac{1}{W_1}P-\lambda)u_n||\\\\
&&+|\lambda|\left(\int_\Omega \left|\frac{W_1- W_2}{W_2}\right|^2u_n^2W_2\mathrm{d}\nu\right)^{1/2}\\\\
&\leq& o(||u_n||).
\end{array}$$
Consequently, $u_n$ is also a Weyl sequence for $\frac{1}{W_2}P$, associated to $\lambda$. The hypothesis on the support of $u_n$ now implies that $\lambda$ is in the essential spectrum of $\frac{1}{W_2}P$.

\end{proof}
Theorem \ref{ess_spectrum_equiv_pot} actually allows us to get a first quick proof of Agmon's claim \ref{Agmon}, relying on the results of \cite{DFP}. An alternative proof, which works more generally for improved Hardy inequalities \eqref{improved_H}, will be given in Section 3.

\begin{Cor}

If $\Omega$ is a bounded, $C^2$ domain in $\R^n$, then the essential spectrum of $\delta^{-2}\Delta$ is $[\frac{1}{4},\infty)$.

\end{Cor}

\begin{proof}
Let $W=\frac{1}{4}\left|\frac{\nabla G}{G}\right|^2$ be an optimal weight in the sense of \cite{DFP}, where $G$ is the Green function of $\Delta$ with pole at some fixed point $x_0\in \Omega$. It has been proved in \cite[Example 13.2]{DFP} that as $x\to\partial \Omega$,

$$W\sim \frac{1}{4\delta^2}.$$
Furthermore, according to \cite[Theorem 2.2]{DFP}, the (essential) spectrum of $\frac{1}{W}\Delta$ is $[1,\infty)$. By Theorem \ref{ess_spectrum_equiv_pot}, the essential spectrum of $\delta^{-2}\Delta$ is equal to the essential spectrum of $4\frac{1}{W}\Delta$, and thus is equal to $[\frac{1}{4},\infty)$.

\end{proof}
Another direct application of Theorem \ref{ess_spectrum_equiv_pot} and of the results of \cite{DFP} is to multipolar Hardy inequalities. Let $x_1,\cdots,x_N$, $N\geq2$ be distinct points in $\R^n$, and consider the positive weight

$$W=\left(\sum_{1\leq i<j\leq N}\frac{|x_i-x_j]^2}{|x-x_i|^2|x-x_j|^2}\right).$$
In \cite{CZ}, the following multipolar Hardy inequality was shown:

\begin{equation}\label{multipolar}
\int_{\R^n}|\nabla u|^2\geq\left(\frac{n-2}{N}\right)^2 \int_{\R^n}Wu^2,\qquad\forall u\in C_0^\infty(\R^n).
\end{equation}
It was proved in \cite[Remark B2 and B.3]{DFP} that $\Delta-\left(\frac{n-2}{N}\right)^2W$ is {\em critical}, but the weight $W$ is {\em not optimal}, i.e. does not belong to $\mathcal{O}(\R^n\setminus \{x_1,\cdots,x_N\})$ (one reason is that the constant $\left(\frac{n-2}{N}\right)^2$ in \eqref{multipolar} is not optimal for test functions supported outisde a ball $B(0,R)$, for $R$ large enough). From Theorem \ref{ess_spectrum_equiv_pot}, one can deduce the following result:

\begin{Cor}[Spectrum for multipolar Hardy inequalities]

Denote $C(N)=\frac{N^2}{4(N-1)}$. If $N=2$, then the (essential) spectrum of $\frac{1}{W}\Delta$ is $[\left(\frac{n-2}{2}\right)^2,\infty)$. If $N>2$, then the essential spectrum of $\frac{1}{W}\Delta$ is $[C(N)\left(\frac{n-2}{N}\right)^2,\infty)$, the bottom of the spectrum of $\frac{1}{W}\Delta$ is $\left(\frac{n-2}{N}\right)^2$, and the positive function $v=\prod_{i=1}^N|x-x_i|^{\frac{2-n}{N}}$ is eigenfunction for $\frac{1}{W}\Delta$, associated to the eigenvalue $\left(\frac{n-2}{N}\right)^2$.

\end{Cor}

\begin{Rem}
{\em 
One can actually show, using arguments similar to the one appearing in the proofs of Lemma \ref{Morse} and Proposition \ref{Morse2}, that if $N>2$, $\frac{1}{W}\Delta$ has a finite number of eigenvalues, each with finite multiplicity, belonging to $[\left(\frac{n-2}{N}\right)^2,C(N)\left(\frac{n-2}{N}\right)^2)$. This in turn implies (see \cite{Baptiste}) that the eigenspace associated to $C(N)\left(\frac{n-2}{N}\right)^2$ is finite dimensional.
}

\end{Rem}
\begin{proof}
Denote $\mathcal{L}:=\frac{1}{W}\Delta$. Let $\varepsilon>0$ be small enough such that the balls $B(x_j,\varepsilon)$ are disjoint. Let $K_1$ be a regular compact set containing the balls $B(x_j,\varepsilon)$, $j=1,\cdots, N$, and denote $K_2$ the complement in $K_1$ of $\cup_{i=1}^NB(x_j,\varepsilon)$. The decomposition principle for the essential spectrum (see \cite{DL} or \cite{Gl}) implies that the essential spectrum of $\mathcal{L}$ is equal to the essential spectrum of $\mathcal{L}$ on $L^2(\R^n\setminus K_1,W\mathrm{d}x)$, with Neumann boundary conditions on $\partial K_1$. Thus, it consists of the union of the essential spectrum of $\mathcal{L}$ on $B(x_j,\varepsilon)$, $j=1,\cdots,N$ with Neumann boundary conditions, and of the essential spectrum of $\mathcal{L}$ on $\R^n\setminus K_1$, with Neumann boundary conditions on $\partial K_1$. When $x\to x_j$, one has (cf \cite[Remark B.3]{DFP})

$$\left(\frac{n-2}{N}\right)^2W\sim C(N)^{-1}\frac{C_H}{|x-x_j|^2}=C(N)^{-1}W_{opt,j},$$
where $C_H:=\left(\frac{n-2}{2}\right)^2$. Since $W_{opt,j}$ is an optimal weight, by \cite[Theorem 2.2]{DFP}, the essential spectrum of $W_{opt,j}^{-1}\Delta$ is $[1,\infty)$. More precisely, for every $\lambda\in [1,\infty)$, one can find a Weyl sequence for $W_{opt,i}^{-1}\Delta$ supported in $B(x_j,\frac{\varepsilon}{2})$. This implies that the essential spectrum of $W^{-1}_{opt,j}\Delta$ on $B(x_j,\varepsilon)$, with Neumann boundary conditions, is equal to $[1,\infty)$. Applying Theorem \ref{ess_spectrum_equiv_pot}, one obtains that the essential spectrum of $\mathcal{L}$ on $B(x_j,\varepsilon)$, $j=1,\cdots,N$ with Neumann boundary conditions is equal to $[C(N)\left(\frac{n-2}{N}\right)^2,\infty)$. Moreover, when $|x|\to\infty$, it holds that

$$W\asymp\frac{1}{|x|^4},$$
and therefore

$$\lambda_\infty(\R^n\setminus K_1,\Delta,W)=\infty,$$
which implies, according to Persson's formula, that $\mathcal{L}$ on $\R^n\setminus K_1$, with Neumann boundary conditions on $\partial K_1$, has no essential spectrum. Therefore, the essential spectrum of $\mathcal{L}$ on $\R^n$ is $[C(N)\left(\frac{n-2}{N}\right)^2,\infty)$. The statement that if $N>2$, the bottom of the spectrum of $\mathcal{L}$ is $\left(\frac{n-2}{N}\right)^2$ and that $v$ is eigenfunction follows from \cite[Remark B.2]{DFP}.

\end{proof}
However, it is more delicate to use Theorem \ref{ess_spectrum_equiv_pot} and \cite[Theorem 2.2]{DFP} in order to prove Theorems \ref{convex_domain_spec} and \ref{minimal_spec} in a similar way. More precisely, in order to prove Theorem \ref{minimal_spec}, one would have to show that as $x\to\infty$ in $M$,

\begin{equation}\label{Green1}
\frac{1}{4}\left|\frac{\nabla G(x)}{G(x)}\right|^2\sim \left(\frac{n-2}{2}\right)^2\frac{1}{r^2(x)},
\end{equation}
where $G$ is the Green function of the Laplacian on $M$, with pole at some fixed point of $M$. In the case of Theorem \ref{convex_domain_spec}, one would have to show that, as $x\to\partial \Omega$,

\begin{equation}\label{Green2}
\frac{1}{4}\left|\frac{\nabla G_{i-1}(x)}{G_{i-1}(x)}\right|^2\sim \mathcal{J}_i(x),
\end{equation}
where $G_{i-1}$ is the Green function of $\Delta-\mathcal{W}_{i-1}$, with pole at some fixed point of $\Omega$. We do not know if the two estimates \eqref{Green1} and \eqref{Green2} hold. In Section 3 and 4, we will prove respectively Theorem \ref{convex_domain_spec} and \ref{minimal_spec}, {\em without having to prove gradient estimates for Green functions}. This is an important point, since in general finding the asymptotic at infinity of $|\nabla \log G|$ for a Green function $G$ is a difficult task, that can be achieved only in some particular cases. Instead, in our general approach that will lead to the proof of Theorems \ref{convex_domain_spec} and \ref{minimal_spec}, we will not have to estimate the gradient of a Green function. This general approach is based on the results of the next subsection.

\subsection{Another spectral result}
We first introduce some definitions and notations. Denote by $\mathcal{C}$ the set of positive measures on $[0,\infty)$ which are absolutely continuous with respect to the Lebesgue measure, and have infinite mass. For $\chi\in\mathcal{C}$ and $r\geq0$, denote

$$V(r):=\int_0^r \mathrm{d}\chi$$
and

$$S(r):=\int_r^{r+1}\mathrm{d}\chi.$$
Also, define $\sigma$, the \textit{exponential rate of volume growth} of $\chi$ by

$$\sigma(\chi):=\lim_{r\to\infty}\sup \frac{1}{r}\log V(r).$$
We have the following elementary technical Lemma:

\begin{Lem}\label{expo_growth}

If $\sigma(\chi)=0$, then for every $a>0$, $\varepsilon>0$ and $d>0$, there exists $b>a$ such that $|b-a|\geq d$ and
 
$$\frac{S(b)}{V(b)-V(a)}<\varepsilon.$$
Conversely, if for every $\varepsilon>0$, there exists $a>0$ and $d>0$ such that for every $b>a+d$, the inequality

$$\frac{S(b)}{V(b)-V(a)}<\varepsilon$$
holds, then $\sigma=0$.

\end{Lem}
\begin{proof}
The proof is elementary, but we provide it for the sake of completeness. Assume that $\sigma=0$. We proceed by contradiction: define $f(r):=V(r)-V(a)$, and assume that there is $\varepsilon>0$ and $R$ big enough such that for every $r\geq R$,

$$\frac{f(r+1)-f(r)}{f(r)}>\varepsilon.$$
Define $g(r):=e^{\nu r}$, with $\nu>0$ chosen so that $e^{\nu}-1<\varepsilon$, then

$$\frac{g(r+1)-g(r)}{g(r)}=e^{\nu}-1<\varepsilon<\frac{f(r+1)-f(r)}{f(r)}.$$
From this, we deduce at once that there is a constant $C$ such that for every $r>0$,

$$f(r)\geq C g(r).$$
But this implies that

$$\sigma \geq \nu,$$
which is impossible.\\
Conversely, if $\varepsilon>0$ is fixed and for every $r>R$,

$$\frac{f(r+1)-f(r)}{f(r)}<\varepsilon,$$
then proceeding as above, by comparison with $g(r):=e^{\nu r}$ with $\nu$ chosen so that $e^{\nu}-1>\varepsilon$, we find that

$$f(r)\leq C e^{\nu r},$$
which yields that $\sigma\leq \nu$. Letting $\varepsilon\to0$, we can let $\nu\to0$ and conclude that $\sigma=0$.

\end{proof}

\begin{Def}
{\em
 A positive measure $\chi\in \mathcal{C}$ is said to have \textit{subexponential volume growth} if it satisfies $\sigma(\chi)=0$. In other words, $\chi$ has subexponential volume growth iff $V(r)=e^{o(r)}$.
}
\end{Def}
We will consider push-forward measures, of which we recall the definition:

\begin{Def}
{\em

If $X$ and $Y$ are measured spaces, $f:X\to Y$ is measurable and $\mu$ is a measure on $X$, then the \textit{push-forward} of $\mu$ by $f$ is the measure $f_\star \mu$, defined so that for every measurable set $B$ of $Y$,

$$(f_\star \mu)(B)=\mu(f^{-1}(B)),$$
where by definition $f^{-1}(B)$ is the subset of $X$ defined by 

$$f^{-1}(B)=\{x\in X \,:\,f(x)\in B\}.$$
Then, for every $g : Y\to\R$ measurable, we have the \textit{change of variable formula}
\begin{equation}\label{ch_variable}
\int_Y g\,\mathrm{d}(f_\star \mu)=\int_X g\circ f \mathrm{d}\mu.
\end{equation}
}
\end{Def}
After these preliminaries, let us now turn to our first general spectral result. Let $L$ be a symmetric, second-order elliptic operator of the form \eqref{div_P}. We assume that $L\mathbf 1=0$ (that is, $c=0$), then by Allegretto-Piepenbrink theory (see \cite{Ag4}, or Lemma 3.10 in \cite{PRS}), $L$ extends to a self-adjoint operator on $L^2(\Omega,\mathrm{d}\nu)$, whose spectrum is included in $[0,\infty)$. In the following Proposition, we give conditions guaranteeing that the spectrum of $L$ is the whole $[0,\infty)$. 

\begin{Pro}\label{spectrum}
Assume that (outside a compact set) there exists a function $v\in C^{1,\alpha}_{loc}$, $\alpha\in (0,1)$ such that, for some positive constant $C$,

$$\lim_{x\to\infty}|\nabla v(x)|_A\to C,$$
and such that

\begin{equation}\label{H1}
\lim_{x\to \infty}v(x)=+\infty
\end{equation}
and
\begin{equation}\label{H2}
\lim_{x\to\infty}Lv(x)=0\,\,\hbox{ (pointwise)}.
\end{equation}
Assume also that the push-forward measure $\chi:=v_\star \nu$ is in $\mathcal{C}$ and has subexponential volume growth. Then the spectrum of $L$ is $[0,\infty)$: more precisely, for every $\eta\geq 0$, we can construct a Weyl sequence whose support goes to infinity, i.e. a sequence $(\varphi_n)_{n\in \mathbb{N}}$ of smooth, compactly supported functions such that 

$$\lim_{n\to\infty}\frac{||(L-\eta)\varphi_n||_2}{||\varphi_n||_2}=0$$
and for every compact $K\subset \Omega$, there is an $N$ such that the support of $\varphi_n$ is in $\Omega\setminus K$ for every $n\geq N$.

\end{Pro}

\begin{Rem}
{\em
\begin{enumerate}
\item The infinity is the ideal point in the one-point compactification of $\Omega$.

\item The condition that $v_\star \nu$ be in $\mathcal{C}$ is quite weak: for example, if $\nu$ is the Lebesgue measure on $\R^n$, then by the co-area formula,  $v_\star \nu\in \mathcal{C}$ if and only if the function $f(t):=\mathrm{Vol}_{\mathcal{H}^{n-1}}(\{v=t\})$ is locally $L^1$ and $\int^\infty f(t)\mathrm{d}t=\infty$. Here, $\mathrm{Vol}_{\mathcal{H}^{n-1}}$ is the volume with respect to $\mathcal{H}^{n-1}$, the $(n-1)-$dimensional Hausdorff measure.

\item As follows from the proof, Proposition \ref{spectrum} holds under the weaker condition on $\chi$ that for every $a>0$, $\varepsilon>0$ and $d>0$, there exists $b>a$ such that $|b-a|\geq d$ and
 
$$\frac{S(b)}{V(b)-V(a)}<\varepsilon.$$
This is indeed (slightly) weaker than subexponential volume growth, by Lemma \ref{expo_growth}.

\end{enumerate}
}
\end{Rem}
The idea behind Proposition \ref{spectrum} is that $L$, restricted to the set of ``radial'' functions, that is functions of the form $f(v)$, has spectrum $[0,\infty)$. Thus it can be thought of in a way as a result about a ``one-dimensional operator''. In fact, for every $\eta\geq0$, the function $e^{i\sqrt{\eta}}$ is an approximate solution at infinity of $(L-\eta)u=0$, and the Weil sequence will be constructed as a sequence of functions that approwimate $e^{i\sqrt{\eta}}$. Let us turn to the detailed proof of Proposition \ref{spectrum}:
\begin{proof} 
Without loss of generality, we can assume that $|\nabla v(x)|_A\to1$ when $x\to \infty$.  We will use the fact that since $L\mathbf1=0$, $L$ satisfies the two formulae (in the weak sense):

\begin{equation}\label{composition}
L(f(u))=f'(u)Lu-f''(u)|\nabla u|_A^2
\end{equation}
and if $g$ and $h$ are $C^{1,\alpha}_{loc}(\Omega)$,

\begin{equation}\label{multiplication}
L(gh)=gL(h)+hL(g)-2 A\nabla g\cdot\nabla h.
\end{equation}
Fix $\eta\geq 0$, and let $\mu:=\sqrt{\eta}.$ Define $\varphi=e^{i\mu v}$. Then, using formula \eqref{composition}, we get that in the weak sense,

$$L\varphi=i\mu  (Lv)\varphi+\mu^2|\nabla v|_A^2 \varphi,$$
that is

$$(L-\eta)\varphi=i\mu (Lv)\varphi+\eta(|\nabla v|_A^2-1)\varphi.$$
We want to define $\varphi_n:=\psi_n(v)\varphi$, where $\psi_n(v)$ is going to play the role of a ``radial'' cut-off function. We first compute (using fomulae \eqref{composition} and \eqref{multiplication}) that in the weak sense,

\begin{equation}\label{smooth_case}
\begin{array}{rcl}(L-\eta)\varphi_n&=&\psi_n(v)\left[(L-\eta)\varphi\right]+\varphi (L\psi_n(v))-2\langle A\nabla \varphi,\nabla \psi_n(v)\rangle\\\\
&=&i\mu\varphi_n(Lv)+\eta(|\nabla v|_A^2-1)\varphi_n\\\\
&&+\varphi\left(\psi_n'(v)(Lv)-\psi_n''(v)|\nabla v|_A^2\right)-2i\mu\psi_n'(v)\varphi|\nabla v|_A^2
\end{array}
\end{equation}
We now define the real function $\psi_n$: we take $\psi_n(t)$ equal to $1$ if $t\in[a_n+1,b_n-1]$, $0$ if $t\notin [a_n,b_n]$ -- $a_n$ and $b_n$ to be chosen later --, and such that there is a constant $C$ independent of $n$ satisfying

$$|\psi_n'|+|\psi''_n|\leq C.$$
We are now ready to estimate each of the terms appearing in the computation of $(L-\eta)\varphi_n$: we first have, using the property of change of variable formula \eqref{ch_variable} of the push-forward measure,

$$||\varphi_n||_2^2=||\psi_n(v)\varphi||^2_2= \int \psi_n^2(t)\mathrm{d}\chi(t)\geq\int_{a_n+1}^{b_n-1}\mathrm{d}\chi(t).$$
Moreover, using that $Lv$ and $|\nabla v|_A$ are bounded, and the fact that $\psi'_n$ and $\psi''_n$ are supported in the union of the intervals $[a_n,a_n+1]\cup[b_n-1,b_n]$, we get (using again the change of variable formula \eqref{ch_variable})

$$||\varphi\left(\psi_n'(v)(Lv)-\psi_n''(v)|\nabla v|_A^2\right)-2i\mu \psi_n'(v)|\nabla v|_A^2\varphi||_2\leq C\int_{[a_n,a_n+1]\cup[b_n-1,b_n]}\mathrm{d}\chi(t).$$
Furthermore, since $|\nabla v|_A^2\to 1$ at infinity, we have

$$||\eta(|\nabla v|_A^2-1)\psi_n(v)\varphi||_2^2=o\left(\int \psi_n^2(t)\mathrm{d}\chi(t)\right)\hbox{ when }n\to\infty.$$ 
We now choose the sequences $(a_n,b_n)_{n\in\mathbb{N}}$ inductively: suppose that $(a_{n-1},b_{n-1})$ is defined, then take $a_n>\min(b_{n-1},n)$ such that

$$|Lv|\leq \frac{1}{n}$$
on the set $\{v\geq a_n\}$ (here we use hypotheses \eqref{H1} and \eqref{H2}), and take $b_n$ big enough such that

$$\frac{\int_{[a_n,a_n+1]\cup[b_n-1,b_n]}\mathrm{d}\chi(t)}{\int_{a_n+1}^{b_n-1}\mathrm{d}\chi(t)}\leq \frac{1}{n}.$$
This is possible by Lemma \ref{expo_growth}, since $\chi$ has subexponential volume growth. Collecting all the estimates, we get that

$$\lim_{n\to\infty}\frac{||(L-\eta)\varphi_n||_2}{||\varphi_n||_2}=0,$$
which concludes the proof.

\end{proof}

\subsection{Spectral result for Hardy inequalities}
We now apply Proposition \ref{spectrum} to study Hardy inequalities. We consider a symmetric operator $P$ on $L^2(\Omega,\mathrm{d}\nu)$ of the form \eqref{div_P}. We take $u_0$, $u_1$ positive functions on $\Omega$ which, for some $\alpha\in(0,1)$ and $K\subset \subset \Omega$ compact subset of $\Omega$, are $C^{1,\alpha}_{loc}(\Omega\setminus K)$. Denote

$$V_i=\frac{Pu_i}{u_i}.$$
Recall that $X_1$ is defined by

$$X_1(t):=\left(1-\log(t)\right)^{-1},$$
and consider the non-negative weight

$$W(u_0,u_1):=\frac{1}{4}\left|\nabla \log\left(\frac{u_0}{u_1}\right)\right|_A^2=\frac{1}{4}\left|\nabla X_1^{-1}\left(\frac{u_0}{u_1}\right)\right|_A^2$$
We emphasize that here and everywhere in the paper, $X_1^{-1}$ is a notation for $\frac{1}{X_1}$ and \textit{not} for the inverse of $X_1$. We know from the \textit{supersolution construction} of \cite{DFP} that in the case where $u_0$, $u_1$ are solutions of $P$, then $u_{1/2}:=\sqrt{u_0u_1}$ and $u_{1/2}X_1^{-1}\left(\frac{u_0}{u_1}\right)$ are solutions of $P-W(u_0,u_1)$. Furthermore, if $u_0$ and $u_1$ are only supersolutions of $P$, then $u_{1/2}$ is supersolution of $P-W(u_0,u_1)$. In this subsection, we will be interested in the case where $u_0$ and $u_1$ are \textit{approximate solutions} of $P$. In the rest of this subsection, we will denote $W(u_0,u_1)$ by $W$. We will need the following general computational lemma (see \cite{DFP} for the proof of the first equality):

\begin{Lem}\label{supconstruct}
The following equalities hold, in the weak sense:

$$\left(P-\frac{1}{2}(V_0+V_1)-W\right)u_{1/2}=0,$$
and


$$\left(P-\frac{1}{2}(V_0+V_1)+(V_0-V_1)X_1\left(\frac{u_0}{u_1}\right)-W\right)u_{1/2}X_1^{-1}\left(\frac{u_0}{u_1}\right)=0.$$

\end{Lem}

For the rest of this section, we assume that $W$ is \textit{positive} in a neighborhood of infinity in $\Omega$: more precisely, we will assume that $W>0$ on $\Omega\setminus K$. From the assumption that $u_0$ and $u_1$ belong to $C^{1,\alpha}(\Omega\setminus K)$, it follows that $W$ and $\frac{1}{W}$ are continuous, and in particular locally bounded, on $\Omega\setminus K$. In order to study the spectral properties of $\frac{1}{W}P$, we perform simultaneously a \textit{h-transform} and a \textit{change of measure}: we consider the operator

$$L:=u_{1/2}^{-1}\left(W^{-1}P-1\right)u_{1/2},$$
which is symmetric on $L^2(\Omega,u_{1/2}^2W\mathrm{d}\nu)$ and unitarily equivalent to $\left(W^{-1}P-1\right)$. We compute from the formulae \eqref{h-transform} and \eqref{change_measure} that

$$L=-\div\left(\frac{A}{W}\nabla\cdot\right)+\frac{\left(W^{-1}P-1\right)u_{1/2}}{u_{1/2}},$$
where the divergence is for the measure $u_{1/2}^2W\mathrm{d}\nu$. Let us denote by $V$ the potential 

$$V:=\frac{\left(W^{-1}P-1\right)u_{1/2}}{u_{1/2}}=\frac{1}{2W}(V_0+V_1),$$
(we have used here the first equality in Lemma \ref{supconstruct}), and by $\tilde{L}$ the symmetric operator

$$\tilde{L}:=-\div\left(\frac{A}{W}\nabla\cdot\right)$$
acting on $L^2(\Omega,u_{1/2}^2W\mathrm{d}\nu)$, so that

$$L=\tilde{L}+V.$$

We have then the following consequence of Proposition \ref{spectrum}:

\begin{Pro}\label{spectrum2}
Assume that the following conditions are satisfied:

\begin{enumerate}

\item $$\lim_{x\to\infty}\frac{u_0(x)}{u_1(x)}=0,$$

\item $$\lim_{x\to\infty}\frac{1}{2W}(V_0+V_1)X_1^{-1}\left(\frac{u_0(x)}{u_1(x)}\right)=0,$$

\item $$\lim_{x\to\infty}\frac{V_0-V_1}{W}=0,$$

\item The push-forward measure $\left(X_1^{-1}\left(\frac{u_0}{u_1}\right)\right)_\star\left(u_0u_1W\mathrm{d}\nu\right)$ is in $\mathcal{C}$ and has subexponential volume growth.

\end{enumerate}
Then the essential spectrum of $W^{-1}P$ on $L^2(\Omega,W\mathrm{d}\nu)$ is $[1,+\infty)$. 

\end{Pro}

\begin{Rem}\label{essential_perturb}

{\em 
\begin{enumerate}
\item It will be clear from the proof that if $\tilde{W}=W$ in a neighborhood of infinity of $\Omega$, then the essential spectrum of $\tilde{W}^{-1}P$ is also $[1,\infty)$.

\item Concerning the hypotheses made in Proposition \ref{spectrum2}: condition (1) expresses the fact that $u_0$ has ``minimal growth'' at infinity; conditions (2) and (3) express in a quantitative way that $u_0$ and $u_1$ are ``approximate solutions'' of $P$ at infinity; condition (4) is satisfied for optimal potentials obtained by the supersolution construction, i.e. if $u_0$ and $u_1$ are \textit{solutions} of $P$: indeed, in this case, $\left(X_1^{-1}\left(\frac{u_0}{u_1}\right)\right)_\star\left(u_0u_1W\mathrm{d}\nu\right)$ has \textit{linear} volume growth (see \cite{DFP}). Therefore, the hypotheses (1)--(4) can be considered to express in a quantitative way that $W$ is ``optimal at infinity''. We will discuss with greater details the relevance of condition (4) in Section \ref{Sec:Agmon_metric}.
\end{enumerate}
}
\end{Rem}
\begin{proof}
Since $W^{-1}P-1$ and $L$ are unitarily equivalent, it is enough to show that the essential spectrum of $L$ is $[0,\infty)$. By Lemma \ref{supconstruct}, we have 

$$LX_1^{-1}\left(\frac{u_0}{u_1}\right)=\frac{1}{2W}(V_0+V_1)X_1^{-1}\left(\frac{u_0}{u_1}\right)+\frac{V_0-V_1}{W}.$$
Define $v:=X_1^{-1}\left(\frac{u_0}{u_1}\right)$, which is $C^{1,\alpha}_{loc}(\Omega\setminus K)$ for some $K$ compact subset of $\Omega$. The hypotheses made imply that

$$\lim_{x\to\infty}v(x)=+\infty,$$

$$\lim_{x\to\infty}V(x)v(x)=0,$$
in particular,
$$\lim_{x\to\infty}V(x)=0,$$
and finally,

$$\lim_{x\to\infty}\tilde{L}v(x)=0.$$
Let us remark that by definition of $W$,

$$|\nabla v|_{A/W}^2=W^{-1}|\nabla v|_A^2=4.$$
We can apply Proposition \ref{spectrum} to $\tilde{L}$ (notice that the matrix $\frac{A}{W}$ appearing in the definition of $\tilde{L}$ is locally bounded in $\Omega\setminus K$): for every $\eta\geq0$, there is a Weyl sequence $(\varphi_n)_{n\in\mathbb{N}}$ for $\tilde{L}-\eta$, with the support of $\varphi_n$ going to infinity. Since $\lim_{x\to\infty}V(x)=0,$ we conclude that $(\varphi_n)_{n\in\mathbb{N}}$ is also a Weyl sequence for $L-\eta$. Hence the essential spectrum of $L$ contains $[0,\infty)$. For the inverse inclusion, it is enough to prove that $0$ is the infimum of the essential spectrum of $L$. By Persson's formula (see \cite{Persson} or \cite{Ag3}), $\lambda_\infty(L)$, the infimum of the essential spectrum of $L$, is given by

\begin{equation}\label{persson}
\lambda_\infty(L)=\sup\{\lambda : \exists K\subset\subset\Omega, \,\exists u_\lambda>0, \,\hbox{s.t. }(L-\lambda)u_\lambda=0\hbox{ on }\Omega\setminus K\}.
\end{equation}
Let $\varepsilon>0$. Since $V$ tends to $0$, we can find a compact set $K_0$ containing $K$ such that

$$|V|\leq \varepsilon\hbox{ on }\Omega\setminus K_0.$$
Since $\tilde{L}\mathbf{1}=0$, by Allegretto-Piepenbrink $\tilde{L}$ is nonnegative. Therefore, again by Allegretto-Piepenbrink one can find a positive function $u$, solution of

$$(L+\varepsilon)u=(\tilde{L}+V+\varepsilon)u=0\hbox{ on }\Omega\setminus{K_0}.$$
Persson's formula \eqref{persson} now implies that $\lambda_\infty(L)\geq-\varepsilon$. Letting $\varepsilon\to0$, we conclude that

$$\lambda_\infty(L)\geq0.$$

\end{proof}
We will be also interested in the part of the essential spectrum \textit{below} the essential spectrum. For this, we will use the following Lemma:

\begin{Lem}\label{Morse}

Let $\mathcal{L}$ be an operator of the form \eqref{div_P}. Assume that in a neighborhood of infinity there is a positive $C^{1,\alpha}_{loc}$ function $v$ such that when $x\to\infty$, there holds

$$|\nabla v(x)|_A\geq C>0$$
and

$$\left|\frac{\mathcal{L}v}{v}+\mathcal{L}\mathbf1\right|=o(v^{-2}).$$
Then the negative spectrum of $\mathcal{L}$ consists of (at most) a finite number of eigenvalues, each with finite multiplicity.

\end{Lem}
\begin{proof}
We apply Lemma \ref{supconstruct} to get, in a neighborhood of infinity,

$$\mathcal{L}\sqrt{v}=\left(\frac{1}{2}\left(\frac{\mathcal{L}v}{v}+\mathcal{L}\mathbf 1\right)+\frac{1}{4}\frac{|\nabla v|_A^2}{v^2}\right)\sqrt{v}.$$
Using the hypotheses, we see that in a neighborhood of infinity,

$$\mathcal{L}\sqrt{v}\geq0.$$
According to \cite{Baptiste}, the existence of a positive supersolution of $L$ outside a compact set of $\Omega$ is equivalent to the fact that the negative spectrum of $L$ consists of (at most) a finite number of eigenvalues, with finite multiplicity, hence the conclusion.

\end{proof}
We apply this in the framework of Proposition \ref{spectrum2}:

\begin{Pro}\label{Morse2}
Let $P$, $W$, $u_0$, $u_1$, $V_0$, $V_1$ be as defined above Proposition \ref{spectrum2}. Assume that when $x\to\infty$, 

\begin{enumerate}

\item $$\lim_{x\to\infty}\frac{1}{2W}(V_0+V_1)=o\left(X_1^2\left(\frac{u_0(x)}{u_1(x)}\right)\right),$$

\item $$\lim_{x\to\infty}\frac{V_0-V_1}{W}=o\left(X_1\left(\frac{u_0(x)}{u_1(x)}\right)\right),$$

\end{enumerate}
Then the spectrum of $W^{-1}P$ strictly below $1$ consists (at most) of a finite number of eigenvalues of finite multiplicity.

\end{Pro}

\begin{Rem}\label{finiteness_perturb}
{\em
It is clear from the proof that the same conclusion also holds for $\tilde{W}^{-1}P$, if $\tilde{W}=W$ in a neighborhood of infinity of $\Omega$.
}
\end{Rem}

\begin{proof}
It is a direct consequence of Lemma \ref{Morse}, applied with $\mathcal{L}=L=u_{1/2}^{-1}(\frac{1}{W}P-1)u_{1/2}$ and $v=X_1^{-1}\left(\frac{u_0(x)}{u_1(x)}\right)$.

\end{proof}
The results of this section lead us naturally to the concept of ``approximate solutions'', as announced in the introduction. We introduce the following definition:

\begin{Def}\label{approx_sol}
{\em 
Recall that $W(u_0,u_1):=\frac{1}{4}\left|\nabla \log\frac{u_0}{u_1}\right|_A^2$, and assume that $W(u_0,u_1)$ is positive in a neighborhood of infinity. We say that $(u_0,u_1)$ is a pair of {\em optimal approximate solutions} at infinity for $P$ if the following conditions are satisfied:

\begin{enumerate}

 \item $$\lim_{x\to\infty}\frac{u_0(x)}{u_1(x)}=0.$$

 \item When $x\to\infty$,
 $$\frac{1}{W(u_0,u_1)}\left(\frac{Pu_0}{u_0}+\frac{Pu_1}{u_1}\right)=o\left(X_1^2\left(\frac{u_0(x)}{u_1(x)}\right)\right).$$
 
 \item When $x\to\infty$,
 $$\frac{1}{W(u_0,u_1)}\left(\frac{Pu_0}{u_0}-\frac{Pu_1}{u_1}\right)=o\left(X_1\left(\frac{u_0(x)}{u_1(x)}\right)\right).$$
 
 \item The push-forward measure $\left(X_1^{-1}\left(\frac{u_0}{u_1}\right)\right)_\star\left(u_0u_1W\mathrm{d}\nu\right)$ is in $\mathcal{C}$ and has subexponential volume growth.

\end{enumerate}
}
\end{Def}

Summarizing the results of this section, we obtain one of the main results of this paper (as a consequence of Propositions \ref{spectrum2} and \ref{Morse2}):

\begin{Thm}\label{general_spec}

Let $P$ be a symmetric, second-order elliptic operator of the form \eqref{div_P}. Let $(u_0,u_1)$ be a pair of optimal approximate solutions at infinity for $P$. Recall that $W(u_0,u_1)=\frac{1}{4}\left|\nabla \log\frac{u_0}{u_1}\right|_A^2$. Then the essential spectrum of the operator $\frac{1}{W(u_0,u_1)}P$ on $L^2(\Omega,W(u_0,u_1)\nu)$ is equal to $[1,\infty)$, and furthermore the spectrum below $1$ consists at most of a finite number of eigenvalues with finite multiplicity.

\end{Thm}

\section{Bounded domains of $\R^n$}\label{Sec:bounded_domain}

\textit{In this section, $\Omega$ will be a general (not necessarily mean convex, otherwise stated) $C^2$ bounded domain of $\R^n$, and $\delta$ is the distance to the boundary of $\Omega$}.\\

In this section, we study the spectrum of the operator $\mathcal{J}_i^{-1}(\Delta-\mathcal{W}_{i-1})$, associated with the Hardy inequality \eqref{improved_H} obtained in \cite{BFT}. We will make use of the following properties of the function $\delta$ (see \cite{MMP} and \cite{Psaradakis}, and references therein):

\begin{Lem}\label{delta}
Let $\Omega\subsetneqq\R^n$ be an open set. Then $\delta$ is Lipschitz in $\Omega$, and

$$|\nabla \delta|=1\mbox{ a.e. on } \Omega.$$
If $\Omega$ is $C^2$, then there exists $U$ neighborhood of $\partial\Omega$ such that $\delta$ is $C^2$ in $U$ (in particular, $|\nabla \delta|=1$ in $U$) and $\Delta\delta$ is bounded in $U$. If $\Omega$ is $C^2$, then $-\Delta\delta\geq0$ in the distribution sense in $\Omega$ if and only if $\Omega$ is mean-convex.

\end{Lem}
The equivalence between mean-convexity and non-negativity of $-\Delta\delta\geq0$ goes back to Gromov, and was established for the first time in \cite{Psaradakis}. \\

Let us recall the definition of the functions $X_i$: we set by convention $X_0\equiv1$,

$$X_1(t):=\left(1-\log(t)\right)^{-1},$$
and for all $i\geq1$,

$$X_{i+1}(t)=X_1(X_i(t)).$$
Then for every $i\geq1$, 

$$X_i((0,1])=(0,1],$$
and
$$X_i(1)=1,\,\,\,\lim_{t\to0}X_i(t)=0.$$
We will use the following formula which gives, for $i\geq1$, the derivative of $X_i$, and which is extracted from \cite{BFT}:

$$X_i'(t)=\frac{1}{t}X_1(t)\cdots X_{i-1}(t)X_i^{2}(t).$$
Now, we have the following computation, which is a direct consequence of Lemma \ref{supconstruct}, and which implies the Hardy inequality \eqref{improved_H} obtained in \cite{BFT} (related computations have been performed in \cite{FT}).
\begin{Pro}\label{Fi-Ter}
Let $i\geq0$, and define

$$U_{0,i}:=\left(\frac{\delta}{D}X_0^{-1}\left(\frac{\delta}{D}\right)\cdots X_i^{-1}\left(\frac{\delta}{D}\right)\right)^{1/2},$$

$$U_{1,i}:=U_{0,i}X_{i+1}^{-1}\left(\frac{\delta}{D}\right),$$

$$R_i:=\sum_{k=1}^iX_1\left(\frac{\delta}{D}\right)\cdots X_k\left(\frac{\delta}{D}\right)\,\,(R_0=0\hbox{ by convention}),$$
and 

$$H_i:=\frac{1}{4}\sum_{k=1}^{i+1}\left|\nabla X_k^{-1}\left(\frac{\delta}{D}\right)\right|^2.$$
Then in $\Omega$,

\begin{equation}\label{form1}
\left(-\Delta-\frac{-\Delta\delta}{2\delta}\left(1-R_i\right)-H_i\right)U_{0,i}=0,
\end{equation}
and

\begin{equation}\label{form2}
\left(-\Delta-\frac{-\Delta\delta}{2\delta}\left(1-R_i\right)+\frac{-\Delta\delta}{\delta}X_1\left(\frac{\delta}{D}\right)\cdots X_{i+1}\left(\frac{\delta}{D}\right)-H_i\right)U_{1,i}=0.
\end{equation}
Moreover, if we denote

$$\mathcal{W}_i:=\frac{1}{4\delta^2}\left(\sum_{k=0}^{i} X_0^2\left(\frac{\delta}{D}\right)\cdots X_k^2\left(\frac{\delta}{D}\right)\right),$$
then 

$$H_i=\mathcal{W}_i \mbox{ a.e. on }\Omega,$$ 
and if $U$ is the neighborhood of $\partial\Omega$ given by Lemma \ref{delta}, then 
$$H_i=\mathcal{W}_i \mbox{ on } U.$$

\end{Pro}

\begin{Cor}\label{Hardy FT}

If $\Omega$ is mean-convex, and if $D$ is chosen such that in $\Omega$, one has

$$R_i=\sum_{k=1}^iX_1\left(\frac{\delta}{D}\right)\cdots X_k\left(\frac{\delta}{D}\right)\leq 1,$$
then the Hardy inequality \eqref{improved_H} takes place on $\Omega$.

\end{Cor}

\begin{Rem}\label{indep of i}
{\em Actually, $D$ in Corollary \ref{Hardy FT} can be chosen independently of $i$: it is a consequence of the fact that the series

$$\sum_{k=0}^\infty X_1\left(t\right)\cdots X_k\left(t\right)$$
converges for every $t\in [0,1)$. For a proof of this fact (kindly provided to us by A. Tertikas), see the Appendix.}

\end{Rem}
\textit{Proof of Corollary \ref{Hardy FT}}. The hypothesis on $D$ gives that $R_i\leq 1$, so that by Proposition \ref{Fi-Ter} and the fact that $-\Delta\delta\geq0$, we have (in the weak sense)

$$\left(-\Delta-H_i\right)U_{0,i}\geq0.$$
Given that $U_{i,0}>0$, this implies by Allegretto-Piepenbrink theory that 

$$-\Delta-H_i\geq0,$$
which is equivalent to saying that the following Hardy inequality holds:

$$\int_\Omega H_i u^2\leq \int_\Omega |\nabla u|^2,\qquad\forall u\in C_0^\infty(\Omega).$$
Since $\delta$ is Lipschitz on $\Omega$, it is not hard to see that $\mathcal{W}_i$ and $H_i$ are in $L^\infty(\Omega)$. By Proposition \ref{Fi-Ter}, there is a set $A\subset\Omega$ of zero measure such that $H_i=\mathcal{W}_i$ on $\Omega\setminus A$. Let us fix $u\in C_0^\infty(\Omega)$; then

$$\int_{\Omega\setminus A} \mathcal{W}_i u^2+\int_AH_iu^2\leq \int_\Omega |\nabla u|^2.$$
Since $\delta$ is Lipschitz on $\Omega$, it is not hard to see that $H_i$ and $\mathcal{W}_i$ are in $L^\infty(\Omega)$, and thus $\int_AH_iu^2=\int_A\mathcal{W}_i=0$. Therefore,

$$\int_{\Omega} \mathcal{W}_i u^2\leq \int_\Omega |\nabla u|^2.$$
Consequently, the Hardy inequality \eqref{improved_H} holds on $\Omega$.

\begin{flushright}
$\Box$
\end{flushright}
\textit{Proof of Proposition \ref{Fi-Ter}}. The proof of the formulae \eqref{form1} and \eqref{form2} is by induction on $i$, using the construction described in Lemma \ref{supconstruct}. Indeed, to pass from \eqref{form1} and \eqref{form2} of index $i$ to \eqref{form1} and \eqref{form2} of index $i+1$, just apply Lemma \ref{supconstruct} with $P=-\Delta+H_i$, $u_0=U_{0,i}$ and $u_1=U_{i,1}$. In order to initialize the induction, we apply Lemma \ref{supconstruct} with $P=-\Delta$, $u_0=\frac{\delta}{D}$, $V_0=\frac{-\Delta\delta}{\delta}$ and $u_1=1$, $V_1=0$, and we get exactly the formulae \eqref{form1} and \eqref{form2} for $i=0$. Now assume that the formulae \eqref{form1} and \eqref{form2} are true for the index $i$, and apply Lemma \ref{supconstruct} with $P=-\Delta+H_i$, $u_0=U_{0,i}$ and $u_1=U_{i,1}$. By the induction hypothesis, 

$$V_0=\frac{-\Delta\delta}{2\delta}\left(1-R_i\right),$$
and

$$V_1=\frac{-\Delta\delta}{2\delta}\left(1-R_i\right)-\frac{-\Delta\delta}{\delta}X_1\left(\frac{\delta}{D}\right)\cdots X_{i+1}\left(\frac{\delta}{D}\right).$$
The formulae of Lemma \ref{supconstruct} gives \eqref{form1} and \eqref{form2} for the index $i+1$, upon noticing that

$$X_1\left(\frac{u_0}{u_1}\right)=X_1\left(X_i\left(\frac{\delta}{D}\right)\right)=X_{i+1}\left(\frac{\delta}{D}\right).$$
The fact that $\mathcal{W}_i=H_i$ on $U$ follows from the following computation:

$$\begin{array}{rcl}
\left|\nabla X_{k+1}^{-1}\left(\frac{\delta}{D}\right)\right|&=&\frac{1}{D}X_{k+1}'\left(\frac{\delta}{D}\right)X^{-2}_{k+1}\left(\frac{\delta}{D}\right)\\\\
&=&\frac{1}{D}\left[\frac{D}{\delta}X_1\left(\frac{\delta}{D}\right)\cdots X_k\left(\frac{\delta}{D}\right)X_{k+1}^2\left(\frac{\delta}{D}\right)\right]X_{k+1}^{-2}\left(\frac{\delta}{D}\right)\\\\
&=&\frac{1}{\delta}X_1\left(\frac{\delta}{D}\right)\cdots X_k\left(\frac{\delta}{D}\right),

\end{array}$$
and therefore $\mathcal{W}_i=H_i$ on $U$.

\begin{flushright}
$\Box$
\end{flushright}
We now prove the announced spectral result for the Hardy inequality \eqref{improved_H}, as a consequence of Proposition \ref{spectrum2}. Recall that for $i\geq1$,

$$\mathcal{J}_i=\mathcal{W}_i-\mathcal{W}_{i-1},$$
and by convention

$$\mathcal{J}_0=\frac{1}{4\delta^2}.$$

\begin{Thm}\label{improved_hardy}

For every $i\geq0$, the {\em essential spectrum} of the operator $\mathcal{L}_i:=\mathcal{J}_i^{-1}(-\Delta-\mathcal{W}_{i-1})$ in $L^2(\Omega, \mathcal{J}_i\,\mathrm{d x})$ is $[1,\infty)$. Furthermore, the spectrum of $\mathcal{L}_i$ strictly below $1$ consists (at most) of a finite number of eigenvalues with finite multiplicity. 

In the case where $\Omega$ is mean-convex and if $D$ is chosen as in Corollary \ref{Hardy FT}, then the {\em spectrum} of $\mathcal{L}_i$ in $L^2(\Omega, \mathcal{J}_i\,\mathrm{d x})$ is $[1,\infty)$, and in particular $1$ is the best constant at infinity in the Hardy inequality \eqref{improved_H}.

\end{Thm}
\begin{proof}
By simplicity, we will assume that

$$H_i=\mathcal{W}_{i-1}$$
on \textit{all} $\Omega$ (and not only in a neighborhood of infinity). If it is not the case, one has to use the Remarks \ref{essential_perturb} and \ref{finiteness_perturb} . The modifications are left to the reader. \\
For $k\geq1$, define $Y_{k}:=X_{k}^{-1}\left(\frac{\delta}{D}\right)$. Fix $i\geq0$. We have seen in the proof of Proposition \ref{Fi-Ter} that the improved Hardy inequality \eqref{improved_H} at step $i$ is obtained by applying the construction of Lemma \ref{supconstruct} with $P=-\Delta-\mathcal{W}_{i-1}$, $u_0=U_{0,i-1}$, $u_1=U_{0,i-1}Y_i$ (for $i=0$, we have to take $P=-\Delta$, $u_0=\frac{\delta}{D}$ and $u_1=1$). For this choice of $u_0$ and $u_1$, we have

$$V_0:=\frac{Pu_0}{u_0}=O(\delta^{-1})$$
and

$$V_1:=\frac{Pu_1}{u_1}=O(\delta^{-1}).$$
We want to apply Proposition \ref{spectrum2} to $P$, $u_0$, $u_1$, $W:=\mathcal{J}_i$. We have to check the three corresponding conditions. Notice that 

$$\frac{u_0}{u_1}=X_{i}\left(\frac{\delta}{D}\right)\to0\hbox{ when }\delta\to0.$$
Next, remark that

$$X_1^{-1}\left(\frac{u_0}{u_1}\right)=X_1^{-1}\left(X_i\left(\frac{\delta}{D}\right)\right)=Y_{i+1}.$$
Therefore,

$$\begin{array}{rcl}
\frac{1}{2\mathcal{J}_i}(V_0+V_1)X_1^{-1}\left(\frac{u_0}{u_1}\right)&=&\frac{1}{2\mathcal{J}_i}Y_{i+1}O(\delta^{-1})\\\\
&=&O(\delta^{1-\varepsilon}).
\end{array}$$
Also,

$$\frac{1}{2\mathcal{J}_i}(V_0-V_1)=O(\delta^{1-\varepsilon}),$$
so the second and the third conditions of Proposition \ref{spectrum2} are satisfied. It remains to verify the condition on the measure. Noticing that 

$$u_0u_1=U_{0,i}^2,$$
we see that the measure is $\mathrm{d}\nu=\mathcal{J}_iU_{0,i}^2\mathrm{d}x$. We compute

$$\begin{array}{rcl}
\mathcal{J}_iU_{0,i}^2&=&\frac{1}{4\delta^2}X_1^2\cdots X^2_i\times \frac{\delta}{D} X_1^{-1}\cdots X_i^{-1}\\\\
&=&\frac{1}{4D}\left(\frac{1}{\delta} X_1\cdots X_i\right)\\\\
&=&\frac{1}{4D}\left|\nabla Y_{i+1}\right|
\end{array}$$
Therefore, by the coarea formula,

$$\begin{array}{rcl}
\int_{\{a\leq Y_{i+1}\leq b\}}\mathcal{J}_iU_{0,i}^2\mathrm{d}x&=&\frac{1}{4D}\int_{\{a\leq Y_{i+1}\leq b\}}|\nabla Y_{i+1}|\mathrm{d}x\\\\
&=&\frac{1}{4D}\int_a^b \left(\int_{\{Y_{i+1}=t\}}\mathrm{d}\sigma\right)\mathrm{d}t
\end{array}$$
But $Y_{i+1}=t$ if and only if $\delta=\varphi(t)$, for a decreasing function $\varphi$ such that

$$\lim_{t\to\infty}\varphi(t)=0.$$
When $\varepsilon\to0$, the surface measure of the level set $\{\delta=\varepsilon\}$ is equivalent to $|\partial\Omega|$, and therefore we obtain that when $a$ and $b$ go to $+\infty$,

$$\int_{\{a\leq Y_{i+1}\leq b\}}\mathcal{J}_iU_{0,i}^2\mathrm{d}x\asymp (b-a).$$
Applying now Proposition \ref{spectrum} gives that the essential spectrum of $\mathcal{L}_i$ is $[1,\infty)$. Concerning the finiteness of the spectrum below $1$, it is immediate to see that the previous computations imply that the hypotheses of Proposition \ref{Morse2} are satisfied, which gives the result.

\end{proof}






The analogue of Proposition \ref{Fi-Ter} and Theorem \ref{improved_hardy} holds for the improved Hardy inequalities considered in \cite{FT}, where $\delta$ is replaced by $\varphi(x):=|x|^{2-n}$, which is \textit{harmonic} (and not only superharmonic). Let us explain this: for $n\geq3$, we consider $\Omega$ a bounded domain of $\R^n$ containing $0$, and we define

$$Z_i:=\left(\varphi(x)X_0^{-1}\left(\frac{|x|}{D}\right)\cdots X_i^{-1}\left(\frac{|x|}{D}\right)\right)^{1/2},$$
we have for $D\geq \sup_\Omega|x|$,

\begin{equation}\label{form12}
\left(-\Delta-\frac{1}{4}\sum_{k=1}^{i+1}\left|\nabla X_k^{-1}\left(\frac{|x|}{D}\right)\right|^2\right)Z_i=0,
\end{equation}
and

\begin{equation}\label{form22}
\left(-\Delta-\frac{1}{4}\sum_{k=1}^{i+1}\left|\nabla X_k^{-1}\left(\frac{|x|}{D}\right)\right|^2\right)Z_iX_{i+1}^{-1}\left(\frac{|x|}{D}\right)=0.
\end{equation}
Equation \eqref{form12} has already been obtained in \cite{FT}. We then have the following result, which is proved exactly like Theorem \ref{improved_hardy}, considering what is happening around zero rather than at the boundary of $\Omega$:

\begin{Thm}
For $n\geq3$, let $\Omega$ is a bounded domain of $\R^n$ containing $0$. Define 

$$H_i:=\frac{1}{4}\sum_{k=1}^{i+1}\left|\nabla X_k^{-1}\left(\frac{|x|}{D}\right)\right|^2$$
and 

$$R_i:=H_i-H_{i-1}.$$
Then the (essential) spectrum of the operator $R_i^{-1}(-\Delta-W_{i-1})$ in $L^2(\Omega,R_i\,\mathrm{dx}$ is $[1,\infty)$. In particular, $1$ is the best constant around zero in the improved Hardy inequality

$$\int_\Omega R_i u^2\leq \int_\Omega|\nabla u|^2-\int_\Omega H_{i-1}u^2,\,\,\forall u\in C_0^\infty(\Omega).$$

\end{Thm}

\begin{Rem}
{\em The fact that $1$ is the best constant around zero was already shown in \cite{FT}.}

\end{Rem}

\section{Minimal immersions of the Euclidean space}\label{Sec:minimal_immersions}
We consider a minimal isometric immersion $M^n\hookrightarrow \R^N$, for $n\geq3$. We will denote by $\mathrm{II}$ the second fundamental form of this immersion. Let $x_0$ be any point of $\R^N$, and let $r=d_{\R^N}(x_0,\cdot)$ be the Euclidean distance. G. Carron \cite{Carron} has shown the following Hardy inequality:

\begin{equation}\label{mini}
\left(\frac{n-2}{2}\right)^2\int_M\frac{u^2}{r^2}\leq \int_M|\nabla u|^2,\,\forall u\in C_0^\infty(M).
\end{equation}
Of course, if $x_0$ is a point of $M$, then an easy argument using test functions localized close to $x_0$ shows that $\left(\frac{n-2}{2}\right)^2$ is the best constant in the Hardy inequality \eqref{mini}. An interesting question is to what extent the weight $W:=\left(\frac{n-2}{2}\right)^2\frac{1}{r^2}$ has the ``best behavior'' at infinity in $M$. We show that

\begin{Thm}\label{optimal_mini}

Let us assume that the total curvature $\int_M|\mathrm{II}|^{n/2}$ is finite. Then the operator $r^2(-\Delta)$ has spectrum $[\left(\frac{n-2}{2}\right)^2,\infty)$ (here, the Laplacian is on $M$). In particular, $\left(\frac{n-2}{2}\right)^2$ is the best constant at infinity in the Hardy inequality \eqref{mini}.

\end{Thm}
This result is not so surprising, since by results of Anderson \cite{Anderson}, we know that the condition on the second fundamental form implies that $M$ is asymptotically Euclidean. In fact, we will use some estimates obtained in \cite{Anderson} in order to establish Theorem \ref{optimal_mini}.

\begin{proof}
The proof is once again an application of Proposition \ref{spectrum2}. First, let us recall Carron's computation that leads to \eqref{mini}. We have 

$$\Delta r^2=n,$$
which implies that

\begin{equation}\label{super_min}
\left(-\Delta-\frac{1}{r^2}\left(\left(\frac{n-2}{2}\right)^2+\frac{(2-n)(n+2)}{4}\left(1-|\nabla r|^2\right)\right)\right)r^{\frac{2-n}{2}}=0.
\end{equation}
Here, $\nabla$ is the gradient on $M$ (and not on $\R^n$). Notice that $|\nabla r|\leq 1$, which implies by the above equation that

$$\left(-\Delta-\left(\frac{n-2}{2}\right)^2\frac{1}{r^2}\right)r^{\frac{2-n}{2}}\geq0.$$
By Allegretto-Piepenbrink theory, this yields the Hardy inequality \eqref{mini}. We will use the following Lemma, consequence of the work of Anderson \cite{Anderson} and Shen-Zhu \cite{SZ}:

\begin{Lem}\label{anderson}
The volume of $M$ is Euclidean at infinity, i.e. if $B_r(x_0)$ is the Euclidean ball centered in $x_0$ and $V(r)$ the volume of $M\cap B_R(x_0)$, then as $r\to \infty$, $V(r)$ is comparable to $r^n$. Also, $M$ has a finite number of ends, and at infinity in each end, $M$ tends to a linear subspace of dimension of $n$; moreover, the second fundamental form satisfies the estimate

$$|\mathrm{II}|=O(r^{-n/2}),$$
and as $r\to\infty$,

$$|\nabla r|^2-1=O(r^{-1}).$$
\end{Lem}
For the estimate of $|\nabla r|^2-1$, see in particular the proof of Lemma 2.4 in \cite{Anderson}.\\

Define  

$$V_0:=\frac{(2-n)(n+2)}{4}\frac{\left(1-|\nabla r|^2\right)}{r^2},$$
then

$$(-\Delta-V_0-W)r^{\frac{2-n}{2}}=0,$$
and by Lemma \ref{supconstruct},

$$\left(-\Delta-\frac{n(n-2)(1-|\nabla r|^2)}{r^2}X_1(r^{2-n})-V_0-W\right)r^{\frac{2-n}{2}}X_1^{-1}(r^{2-n})=0.$$
Furthermore,

$$\frac{1}{W}\left|\nabla X_1^{-1}(r^{2-n})\right|^2\to 1\hbox{ when }r\to\infty.$$
The \textit{proof} of Proposition \ref{spectrum2}, which relies on Proposition \ref{spectrum}, shows that it is enough to prove the following three properties:

\begin{enumerate}

 \item $$\lim_{r\to\infty}\frac{V_0}{W}X_1^{-1}(r^{2-n})=0,$$
 
 \item $$\lim_{r\to\infty}\frac{1}{W}\frac{n(n-2)(1-|\nabla r|^2)}{r^2}=0,$$
 
 \item The measure $r^{2-n}W\mathrm{d}x$ on $M$ satisfies the hypothesis (4) of Proposition \ref{spectrum}. 

\end{enumerate}
The first two claims are consequences of the estimate of $1-|\nabla r|^2$ given by Lemma \ref{anderson}. For the last one, since 

$$X_1^{-1}(r^{2-n})\sim (n-2)\log r\hbox{ when }r\to\infty,$$
we see that it is enough to prove the estimate on the measure with $v=\log r$. By the co-area formula, 

$$\begin{array}{rcl}
\int_{a\leq \log r\leq b}r^{2-n} W&=&\int_{a\leq \log r\leq b}r^{2-n}r^{-2}\\\\
&\asymp&\int_{a\leq \log r\leq b}r^{-n}|\nabla r|\\\\
&\asymp&\int_{e^a}^{e^b}t^{-n}\mathrm{d}V(t),\\\\
\end{array}$$
so that the corresponding measure on $\R_+$ is 

$$\mathrm{d}\chi(t)=t^{-n}\mathrm{d}V(t).$$
We have to check that for any $a$ big enough,

$$\lim_{b\to\infty} \frac{\int_{b-1}^bt^{-n}\mathrm{d}V(t)}{\int_a^b t^{-n}\mathrm{d}V(t)}=0.$$
Integrating by part gives the formula

$$\int_c^d t^{-n}\mathrm{d}V(t)=d^{-n}V(d)-c^{-n}V(c)+n\int_c^d t^{-n-1}V(t)\mathrm{d}t.$$
Using the hypothesis that the volume growth in $M$ is Euclidean at infinity (Lemma \ref{anderson}), we see that

$$\frac{\int_{b-1}^bt^{-n}\mathrm{d}V(t)}{\int_a^b t^{-n}\mathrm{d}V(t)}\asymp \frac{1}{\log(b)},$$
hence goes to $0$ when $b\to\infty$. This gives that $\chi$ is in the class $\mathcal{C}$ of measure with subexponential volume growth.

\end{proof}

\section{Spectrum and Agmon metric}\label{Sec:Agmon_metric}
In this section, we study the relationship between good Hardy inequalities and some weak ``hyperbolicity'' properties. The motivation comes from the following example: in the case of the Euclidean unit ball $B$, consider the Hardy inequality

$$\frac{1}{4}\int_B \frac{u^2}{\delta^2}\leq \int_B|\nabla u|^2,\,\forall u\in C_0^\infty(B).$$
Define the following metric (the associated \textit{Agmon metric}, see the paragraph below):

$$ds^2=\frac{1}{\delta^2}dx^2,$$
then $ds^2$ is the {\em hyperbolic metric} on the ball $B$. We ask the following question:

\begin{Ques}
{\em 
For a general Hardy inequality with an admissible potential, does the corresponding Agmon metric retain some (weak) hyperbolicity properties, i.e. properties similar to the one of the hyperbolic metric $\frac{1}{\delta^2}dx^2$ on $B$?
}
\end{Ques}
We will show in this section that in the case of a general good Hardy inequalities, the Agmon metric satisfies a property of exponential growth of volume, similar to the exponential volume growth of the hyperbolic space. We will then apply this to the example of improved Hardy inequalities \eqref{ex_series_convex} on a mean convex domain of the Euclidean space. In passing, we will show a connection between this property of exponential growth of volume, and the condition (4) (the condition of subexponential volume growth of the measure) in Proposition \ref{spectrum2}.

\subsection{General case}
In this section, we consider a general Hardy inequality. Let $W$ be a \textit{positive} potential, and $P$ of the form \eqref{div_P}, such that the following Hardy inequality takes place for some $\lambda>0$:

\begin{equation}\label{Hardy_sp}
\lambda \int_\Omega Wu^2\mathrm{d}\nu\leq q(u),\,\forall u\in C_0^\infty(\Omega),
\end{equation}
where $q$ is the quadratic form of $P$. Our first result is a direct generalization Brooks' results \cite{Brooks1}, \cite{Brooks2}, who proved an estimate for the bottom of the essential spectrum of the Laplace-Beltrami operator on a complete Riemannian manifold, in term of the exponential growth of the volume of the geodesic balls. Brooks' results are consequence of Agmon's work on exponential decay of solutions of second-order elliptic equations \cite{Ag3}. Here we will show that Brooks' results can be formulated in the more general context of Hardy inequalities: our result is an estimate of the best constant at infinity in the Hardy inequality \eqref{Hardy_sp}, in term of the exponential volume growth of some measure. Before presenting our result, we need to introduce some definitions and notations. Let $\varphi$ be a positive solution of 

$$P\varphi=0,$$
and let us perform a \textit{h-transform} with respect to $\varphi$: define

$$\tilde{P}:=\varphi^{-1}P\varphi,$$
which is self-adjoint on $L^2(\Omega,\varphi^2\mathrm{d}\nu)$. By formula \eqref{h-transform}, $\tilde{P}$ is given by

$$\tilde{P}u=-\div_{\varphi^2}(A\nabla u).$$
Now perform a \textit{change of measure}: introduce the measure

$$\mathrm{d}\mu:=\varphi^2W\mathrm{d}\nu,$$
and define

$$L:=\frac{1}{W}\tilde{P},$$
which is self-adjoint on $L^2(\Omega,\varphi^2W\mathrm{d}\nu)$, so that the Hardy inequality \label{Hardy_sp} is equivalent to

\begin{equation}\label{Hardy_sp2}
\lambda \int_\Omega u^2\,\mathrm{d}\mu\leq \int_\Omega |\nabla u|_{\frac{A}{W}}^2\,\, \mathrm{d}\mu,\,\forall u\in C_0^\infty(\Omega),
\end{equation}
where $|\xi|_{\frac{A}{W}}^2:=\langle \frac{A}{W}\xi,\xi\rangle$. The term on the right-hand side is the quadratic form associated to $L$. Furthermore, the best constant $\lambda_0$ (resp. best constant at infinity $\lambda_\infty$) in \eqref{Hardy_sp} is equal to the best constant (resp. best constant at infinity) in \eqref{Hardy_sp2}. The Hardy inequality \eqref{Hardy_sp2} expresses that $L$ has a spectral gap, indeed as we have already indicated in the introduction, the bottom of the spectrum (resp. essential spectrum) of $L$ is $\lambda_0$ (resp. $\lambda_\infty$).  Let us define the Agmon metric

\begin{equation}\label{metric_Ag}
|\xi|_{Ag}^2:=W\langle A^{-1}\xi,\xi\rangle.
\end{equation}
Denote by $\rho$ the distance function in this Agmon metric. If $\mu(\Omega)=\infty$, then for a fixed $x_0\in \Omega$, denote by $V(r)$ the volume for the measure $\mu$ of the ball $B_{Ag}(x_0,r)$ of center $x_0$ and of radius $r$ (with respect to the distance $\rho$). If $\mu(\Omega)<\infty$, then define $V(r)^{-1}$ to be the volume of $\Omega\setminus B_{Ag}(x_0,r)$ for the measure $\mu$ instead. Finally, define $\sigma$ by

\begin{equation}\label{expo}
\sigma:= \lim_{r\to\infty}\,\sup\frac{1}{r} \log V(r).
\end{equation}
If $\mu(\Omega)=\infty$, then $\sigma$ is the exponential rate of volume growth for the measure $\mu$, in the Agmon metric \eqref{metric_Ag}. If $\mu(\Omega)<\infty$, then $\sigma$ is the exponential rate of convergence to $\mu(\Omega)$ of the volume of balls in the Agmon metric. This definition does not depend on the choice of $x_0$. Our result is

\begin{Thm}\label{Brooks}

Assume that the Agmon metric is complete. Then the following inequality takes place:

$$\lambda_\infty\leq \frac{\sigma^2}{4}.$$

\end{Thm}
\begin{proof} The proof follows closely Brook's proofs in \cite{Brooks1}, \cite{Brooks2}, once the reduction to inequality \eqref{Hardy_sp2} has been made. The need to use the Agmon metric \eqref{metric_Ag} instead of the Riemannian metric on $\Omega$ comes from the following fact, that we extract from \cite{Ag3}, Theorem 1.4:

\begin{Lem}
The distance $\rho$ for the Agmon metric \eqref{metric_Ag} satisfies

$$|\nabla \rho(x_0,\cdot)|_{\frac{A}{W}}^2\leq 1.$$

\end{Lem}
With this at hand, the adaptation of Brook's proof is quite straightforward, replacing the Riemannian metric by the Agmon metric and the Riemannian volume form by the measure $\mu$. We leave the details to the reader.

\end{proof}
\begin{Rem}\label{complete_ag}

{\em 
If $P=-\div(A\nabla \cdot)$ and $W=|\nabla h|_A^2$ for some real-valued Lipschitz function $h$, then we have the following formula for the distance in the Agmon metric \eqref{metric_Ag} (for a proof, see Lemma 10.5 in \cite{DFP}):

\begin{equation}\label{dist_ag}
\rho(x,y)=|h(x)-h(y)|.
\end{equation}
The completeness of $\Omega$ for the Agmon metric is then equivalent to 

$$\lim_{x\to\infty}|h(x)|=\infty$$
(see Lemma A1.2 in \cite{Ag3}). In the case of the optimal potentials obtained by the supersolution construction in \cite{DFP}, the Agmon metric is complete (Lemma 10.5 in \cite{DFP}) and the inequality of Theorem \ref{Brooks} is an equality (see Lemma 7.2 in \cite{DFP}).  

}
\end{Rem}
In fact, following the argument of Li and Wang \cite[Theorem 1.3]{Li-Wang1}, more can be said:

\begin{Thm}

For every $\varepsilon>0$, the following holds:

\begin{enumerate}

\item if $\mu(\Omega)<\infty$, then there is a constant $C$ such that

$$V(\infty)-V(r)\leq C\exp\left(-(2-\varepsilon)\sqrt{\lambda_\infty}\right),$$
where $V(\infty):=\mu(\Omega)$.

\item if $\mu(\Omega)=\infty$, then there is a constant $C$ such that

$$V(r)\geq C\exp\left((2-\varepsilon)\sqrt{\lambda_\infty}\right).$$

\end{enumerate}
Furthermore, if the spectrum of $L$ below $\lambda_\infty$ consists of a finite number of eigenvalues, each with finite multiplicity, then we can take $\varepsilon=0$ in the above inequalities.

\end{Thm}
To conclude this general subsection, we present a result about the \textit{discrete} spectrum case, related to Lemma \ref{discrete}. Let $P$ be of the form \eqref{div_P} qith quadratic form $q$, and let $W$ be a positive potential. Define $\mathcal{D}^{1,2}(\Omega,P,W)$ to be the completion of $C_0^\infty(\Omega)$ for the norm

$$\left(q(u)^2+\int_\Omega |u|^2W\mathrm{d}\nu\right)^{1/2}.$$

\begin{Thm}

Assume that for some $\lambda>0$, the Hardy inequality \eqref{hardy} holds. Let $V$ be a positive potential such that

$$\lim_{x\to \infty}\frac{V(x)}{W(x)}=0.$$
Then $\lambda_\infty(\Omega,P,V)=+\infty$. In particular, the spectrum of $V^{-1}P$ consists of an increasing sequence of eigenvalues, tending to $+\infty$, and if $\lambda\in \R$ does not belong to the spectrum of $\frac{1}{W}P$, then the resolvent $(\frac{1}{W}P-\lambda)^{-1}$ is compact. 

If moreover the Agmon metric

$$|\xi|_{Ag}^2:=W\langle A^{-1}\xi,\xi\rangle$$
is complete and $V\in L_{loc}^\infty(\Omega)$, then $\mathcal{D}^{1,2}(\Omega,P,V)$ injects compactly into $L^2(\Omega,V\mathrm{d}\nu)$.

\end{Thm}

\begin{proof}
Let us prove the first part of the theorem. Let $\varepsilon>0$. Then there is a compact set $K$ such that for every $x\in \Omega\setminus K$, 

$$\frac{V(x)}{W(x)}\leq \varepsilon.$$
Therefore, using the fact that the Hardy inequality \eqref{hardy} is satisfied by assumption, we obtain that for every $u\in C_0^\infty(\Omega\setminus K)$,

$$\lambda\varepsilon^{-1} \int_{\Omega\setminus K}Vu^2\,\mathrm{d}\nu\leq \lambda \int_{\Omega\setminus K}Wu^2\,\mathrm{d}\nu\leq q(u).$$
Therefore, 

$$\lambda_\infty(\Omega,P,W)\geq \lambda\varepsilon^{-1}.$$
Letting $\varepsilon\to0$, we get that

$$\lambda_\infty(\Omega,P,W)=+\infty.$$
This concludes the proof of the first part of the theorem.\\

Let us now prove the second part. First, using a \textit{h-transform}, we can assume that $P\mathbf1=0$. We denote in all the proof $\mathcal{D}^{1,2}(\Omega,P,W)=\mathcal{D}^{1,2}$. Fix a point $x_0\in \Omega$, and let $\varepsilon>0$. By hypothesis, there exists a compact set $K_0$ such that for every $x\in \Omega\setminus K_0$,

$$\frac{V(x)}{W(x)}\leq \varepsilon.$$
Let $K$ compact set of $\Omega$, containing $K_0$ in its interior. As a consequence of the completeness of the Agmon metric, there is a constant $C$ (indepedant of $\varepsilon$) such that for every $R>0$ big enough, there exists a $C_0^\infty$ cut-off function $\varphi_R$ equal to $1$ on $B_{Ag}(x_0,R)\setminus K$, zero outside $B_{Ag}(x_0,R+1)$, and zero in $K_0$, such that

$$||\varphi_R||_\infty+|\nabla \varphi_R|_{\frac{A}{W}}\leq C.$$
Let $(w_n)_{n\in \mathbb{N}}$ be a bounded sequence in $\mathcal{D}^{1,2}(\Omega,P,W)$. Up to a subsequence, one can assume that $(w_n)_{n\in \mathbb{N}}$ converges weakly in $\mathcal{D}^{1,2}(\Omega,P,W)$, and by substracting the limit, one can assume that $(w_n)_{n\in \mathbb{N}}$ converges weakly to zero in $\mathcal{D}^{1,2}(\Omega,P,W)$. Let us compute

$$\begin{array}{rcl}
\int_\Omega (A\nabla (\varphi_Rw_n)\cdot \nabla (\varphi_Rw_n))&=&\int_\Omega (A\nabla w_n\cdot \nabla w_n)\varphi_R^2\\\\
&+2&\int_\Omega (A\nabla w_n\cdot \nabla \varphi_R)w_n\varphi_R+\int_\Omega  (A\nabla \varphi_R\cdot \nabla \varphi_R)w_n^2\\\\
&\leq&C \int_\Omega (A\nabla w_n\cdot \nabla w_n)+\int_\Omega  (A\nabla \varphi_R\cdot \nabla \varphi_R)w_n^2\\\\
&&+2C\int_\Omega (A\nabla w_n\cdot \nabla w_n)^{1/2}(A\nabla \varphi_R\cdot \nabla \varphi_R)^{1/2}w_n\\\\
&\leq& C\int_\Omega (A\nabla w_n\cdot \nabla w_n)+C\int_\Omega W w_n^2\\\\
&&+2C(A\nabla w_n\cdot \nabla w_n)^{1/2} W^{1/2}w_n\\\\
&\leq&C\int_\Omega (A\nabla w_n\cdot \nabla w_n)\\\\
&&+2C\left(\int_\Omega (A\nabla w_n\cdot \nabla w_n)\right)^{1/2}\left(\int_\Omega Ww_n^2\right)^{1/2}\\\\
&\leq&C\int_\Omega (A\nabla w_n\cdot \nabla w_n)\\\\
&\leq& C,
\end{array}$$
where $C$ is independent of $R$. Here, we have successively used that $(A\nabla \varphi_R,\nabla \varphi_R)\leq CW$, and the Hardy inequality \eqref{hardy}, satisfied by hypothesis. Also, using the hypothesis on $V$, the definition of $\varphi_R$ and the Hardy inequality \eqref{hardy},

$$\begin{array}{rcl}
\int_{\Omega\setminus K_0}V(\varphi_Rw_n)^2&\leq& \varepsilon \int_{\Omega\setminus K_0}W(\varphi_Rw_n)^2\\\\
&\leq& \varepsilon q(\varphi_R w_n)=\varepsilon \int_\Omega (A\nabla (\varphi_Rw_n)\cdot \nabla (\varphi_Rw_n))
\end{array}$$
(recall that we have assumed that $P\mathbf1=0$). Therefore,

$$\int_{\Omega\setminus K_0}V(\varphi_Rw_n)^2\leq \varepsilon C.$$
Since $C$ is independent of $R$, we get, letting $R\to\infty$:

$$\int_{\Omega\setminus K}Vw_n^2\leq \varepsilon C.$$
By local uniform ellipticity of $P$ and the Rellich theorem, $(w_n)_{n\in \mathbb{N}}$ converges to zero in $L^2_{loc}(\Omega)$. Since $V\in L_{loc}^\infty(\Omega)$,

$$\lim_{n\to\infty}\int_{K}Vw_n^2=0.$$
Finally, we obtain that

$$\limsup_{n\to\infty}\int_{\Omega}Vw_n^2\leq \varepsilon C.$$
Letting $\varepsilon\to0$, we conclude that

$$\lim_{n\to\infty}\int_{\Omega}Vw_n^2=0,$$
and the proof of the second part of the theorem is complete.

\end{proof}

\subsection{The case of approximate solutions and the role of subexponential growth}
In this subsection, we investigate what happens for a Hardy inequality with a weight obtained by the \textit{supersolution construction} of \cite{DFP}. We let $u_0$ and $u_1$ be positive supersolutions of $P$, and we recall the notation $W(u_0,u_1):=\frac{1}{4}\left|\nabla X_1^{-1}\left(\frac{u_0}{u_1}\right)\right|_A^2$. By Lemma 5.1 in \cite{DFP}, the following Hardy inequality takes place for $\lambda=1$:

\begin{equation}\label{Hardy_appr}
\lambda\int_\Omega W u^2\mathrm{d}\nu\leq q(u),\,\,\forall u\in C_0^\infty(\Omega), 
\end{equation}
where $q$ is the quadratic form of $P$, and define $\lambda_0$ and $\lambda_\infty$ to be respectively the best constant and the best constant at infinity in \eqref{Hardy_appr}. Define also the measures $\mu_i$ for $i=0,1$ by

$$\mu_i:=u_i^2W(u_0,u_1)\,\nu,$$
and the measure 

$$\mu:=u_0u_1W(u_0,u_1)\,\nu,$$
where we recall that $\nu$ is the underlying measure. Define the Agmon metric

\begin{equation}\label{metric_Ag2}
|\xi|_{Ag}^2:=W(u_0,u_1)\langle A^{-1}\xi,\xi\rangle,
\end{equation}
and let $\rho$ be the distance for the Agmon metric. By Remark \ref{complete_ag}, $\rho$ is given by formula \eqref{dist_ag} with $h=\frac{1}{2}X_1^{-1}\left(\frac{u_0}{u_1}\right)$. Define also $\sigma_i$ for $i=1,2$, being the exponential rate of volume growth of $\sigma_i$, as in Definition \ref{expo} with $\mu$ replaced by $\mu_i$. Let us first consider an example, which introduces the results that we want to present in this subsection:

\begin{Exa}\label{Exa_vol}
{\em

Let $\Omega$ be a bounded, smooth domain, with $P=-\Delta$, $\nu=\mathrm{dx}$, $W=\frac{1}{4\delta^2}=W(u_0,u_1)$ for $u_0=\delta$, $u_1=1$. Then $\mu_0=\frac{\mathrm{dx}}{4}$, $\mu_1=\frac{\mathrm{dx}}{4\delta^2}$ and $\mu=\frac{\mathrm{dx}}{4\delta}$. Also, the distance in the Agmon metric $\frac{\mathrm{d}x^2}{4\delta^2}$ is 

$$\rho(x,y)=\frac{1}{2}\left|\log\left(\frac{\delta(x)}{\delta(y)}\right)\right|.$$
Moreover, elementary computations show that:

\begin{enumerate}

 \item $\mu_0$ has finite volume, $\mu_1$ has infinite volume.

 \item $1=\lambda_\infty(\Omega,-\Delta,\frac{1}{4\delta^2})=\frac{\sigma_0^2}{4}=\frac{\sigma_1^2}{4}$.

 \item $\mu$ has \textit{linear} volume growth: 
 
 $$\mu(B_{Ag}(x_0,R)) \asymp R,\qquad \forall R\geq 1,$$
 where we recall that $B_{Ag}(x_0,R)$ is the geodesic ball of center $x_0$ and radius $R$ in the Agmon metric  $\frac{\mathrm{d}x^2}{4\delta^2}$.

\end{enumerate}

}

\end{Exa}
In the rest of this subsection, we will show that properties (1), (2) of Example \ref{Exa_vol} hold in more general situations. Let us begin to show an analogue of Theorem \ref{Brooks}:

\begin{Pro}\label{Brooks2}
Assume that $W(u_0,u_1)$ is positive, that

$$\lim_{x\to\infty}\frac{u_0(x)}{u_1(x)}=0,$$
and that for $i=0,1$,

$$\lim_{x\to\infty} \frac{1}{W(u_0,u_1)}\frac{Pu_i}{u_i}=0.$$
Then for $i=1,2$, the inequality

\begin{equation}\label{growth_est}
\lambda_\infty\leq \frac{\sigma_i^2}{4}
\end{equation}
holds.

\end{Pro}
\begin{proof}
The completeness of the Agmon metric follow from Remark \ref{complete_ag} and the hypothesis that

$$\lim_{x\to\infty}\frac{u_0(x)}{u_1(x)}=0.$$
In the rest of the proof, we will denote $W(u_0,u_1)$ by $W$, and we let $V_i:=\frac{Pu_i}{u_i}$. As in section 5.1, we successively perform an \textit{h-transform} and a \textit{change of measure}. But this time, the \textit{h-transform} is performed with respect to an approximate solution of $P-W$, and not to a solution: we define two operators

$$L_i:=\frac{1}{W}u_i^{-1}Pu_i,$$
where $i\in\{0;1\}$, where $L_i$ is self-adjoint on $L^2(\Omega,u_i^2W\mathrm{d}\nu)$. By formulae \eqref{h-transform} and \eqref{change_measure}, we have

$$L_i=-\div_i\left(\frac{A}{W}\cdot\right)+\frac{V_i}{W},$$
where the divergence $\div_i$ is for the measure $\mu_i$.
Denote by $\tilde{L}_i$ the operator

$$\tilde{L}_i:=-\div_i\left(\frac{A}{W}\cdot\right).$$

Under the assumptions on $V_i$ made, an argument involving Persson's formula (see the proof of Proposition \ref{spectrum2}) shows that for $i=1,2$, the bottom of the essential spectrum of $L_i$ is equal to the bottom of the essential spectrum of $\tilde{L_i}$. But since $L_i$ is unitarily equivalent to $\frac{1}{W}P$, its bottom of the spectrum is $\lambda_\infty$, and therefore the bottom of the spectrum of $\tilde{L}_i$ is $\lambda_\infty$. We can now apply Theorem \ref{Brooks} to $\tilde{L}_i$ to get the result.

\end{proof}
We now turn to the reverse inequalities for $\sigma_0$, $\sigma_1$. This, as we shall see, requires conditions on the growth of the measure $\mu=u_0u_1 W(u_0,u_1)\,\nu$. These conditions generalize Property (3) (linear growth of $\mu$) of Example \ref{Exa_vol}, and are linked with the condition (4) appearing in Proposition \ref{spectrum2}. Denote by $\chi$ the push-forward of the metric $\mu$ by $X_1^{-1}\left(\frac{u_0}{u_1}\right)$:

$$\chi:=\left(X_1^{-1}\left(\frac{u_0}{u_1}\right)\right)_\star\mu.$$
The measure $\chi$ is the measure which appears in the condition (4) of Proposition \ref{spectrum2}. 

\begin{Def}

{\em

We will say that $\mu$ has {\em at most $\varepsilon-$exponential growth and decay} if there exists a constant $C>0$ such that for every $t>0$,

$$C^{-1}e^{-\varepsilon t}\leq \chi(t)\leq C\,e^{\varepsilon t}$$

}

\end{Def}

\begin{Rem}
{\em

The inequality $\chi(t)\leq C\,e^{\varepsilon t}$ is related -- but not equivalent -- to subexponential growth for $\chi$ (that is, is related to the condition (4) in Proposition \ref{spectrum2}). For example, $\mu=\frac{\mathrm{d}x}{4\delta}$ of Example \ref{Exa_vol}, has $\varepsilon-$exponential growth and decay, for every $\varepsilon>0$.
}
\end{Rem}
With this definition, we show the following result:

\begin{Pro}\label{reverse_ineq}

Assume that $\frac{u_0}{u_1}$ is bounded from above, that

$$\lim_{x\to\infty}\frac{u_0(x)}{u_1(x)}=0,$$
and that for some $0<\varepsilon<1$, $\mu$ has at most $\varepsilon-$exponential growth and decay. Then $\mu_0(\Omega)<\infty$, $\mu_1(\Omega)=\infty$, and for $i=0,1$, the reverse inequality

$$\frac{\sigma_i^2}{4}\leq 1$$
holds.

\end{Pro}
As a direct consequence of Proposition \ref{Brooks2} and Proposition \ref{reverse_ineq}, we get the following corollary:
\begin{Cor}\label{equality_case}

Assume that $\frac{u_0}{u_1}$ is bounded from above, that assumptions of Proposition \ref{spectrum2} are satisfied, that there is $0<\varepsilon<1$ and some constant $C$ such that $\mu$ has at most $\varepsilon-$exponential growth and decay, and moreover that for $i=0,1$,

$$\lim_{x\to\infty} \frac{1}{W(u_0,u_1)}\frac{Pu_i}{u_i}=0.$$
Then $\mu_0(\Omega)<\infty$, $\mu_1(\Omega)=\infty$, the Agmon metric \eqref{metric_Ag2} is complete and for $i=0,1$,

$$\frac{\sigma_i^2}{4}=\lambda_\infty=1.$$

\end{Cor}

\begin{proof}
By normalization of $u_0$ and $u_1$, we will assume without loss of generality that

$$\frac{u_0}{u_1}\leq 1.$$
We will denote $W:=W(u_0,u_1)$ and $V_i:=\frac{Pu_i}{u_i}$. Let us also denote by $\psi$ the inverse function of $X_1^{-1}$. Since 

$$X_1^{-1}(t)\sim -\log(t) \hbox{ as }t\to0,$$
we have

$$\psi(t)\sim e^{-t}\hbox{ as }t\to0.$$
Now, 

$$\psi\left(X_1^{-1}\left(\frac{u_0}{u_1}\right)\right) \mu=u_0^2 W\,\nu=\mu_0,$$
and

$$\left(\frac{1}{\psi}\right)\left(X_1^{-1}\left(\frac{u_0}{u_1}\right)\right) \mu=u_1^2 W\,\nu=\mu_1.$$
Thus, using the change of variable formula \eqref{ch_variable}, we see that

$$\mu_0\left(a\leq X_1^{-1}\left(\frac{u_0}{u_1}\right)\leq b\right)=\int_a^b\psi(t)\mathrm{d}\chi(t),$$
and

$$\mu_1\left(a\leq X_1^{-1}\left(\frac{u_0}{u_1}\right)\leq b\right)=\int_a^b\frac{\mathrm{d}\chi(t)}{\psi(t)}.$$
By the hypothesis that

$$\lim_{x\to\infty}\frac{u_0}{u_1}=0,$$
we have

$$\mu_0(\Omega)=\mu_0\left(1\leq X_1^{-1}\left(\frac{u_0}{u_1}\right)<\infty\right)=\int_1^\infty\psi(t)\mathrm{d}\chi(t),$$
and

$$\mu_1(\Omega)=\mu_1\left(1\leq X_1^{-1}\left(\frac{u_0}{u_1}\right)<\infty\right)=\int_1^\infty\frac{\mathrm{d}\chi(t)}{\psi(t)}.$$
Since $\psi(t)\sim e^{-t}$ as $t\to\infty$, $\mu_0(\Omega)<\infty$ (resp. $\mu_1(\Omega)=\infty)$ is equivalent to $\int_1^\infty e^{-t} \mathrm{d}\chi(t)<\infty$ (resp. $\int_1^\infty e^{t} \mathrm{d}\chi(t)=\infty$). But by the integration by part formula, valid for Stieljes measures,

$$\int_1^\infty e^{-t} \mathrm{d}\chi(t)=\left[ e^{-t}\chi(t)\right]_1^\infty+\int_1^\infty e^{-t}\chi(t)\mathrm{d}t,$$
and given the hypothesis on $\chi$, 

$$\lim_{t\to\infty}e^{-t}\chi(t)=0,$$
and

$$\int_1^\infty e^{-t}\chi(t)\mathrm{d}t<\infty.$$
This proves that $\mu_0(\Omega)<\infty$. For $\mu_1(\Omega)$, we have again by integration by parts

$$\int_1^\infty e^{t} \mathrm{d}\chi(t)=\left[ e^{t}\chi(t)\right]_1^\infty+\int_1^\infty e^{t}\chi(t)\mathrm{d}t,$$
and given the hypothesis on $\chi$,

$$\int_1^\infty e^{t}\chi(t)\mathrm{d}t=\infty,$$
which yields $\mu_1(\Omega)=\infty.$ Now we turn to the estimates on $\sigma_0$ and $\sigma_1$. Since the Agmon metric is given by formula \eqref{dist_ag} with $h=X_1^{-1}\left(\frac{u_0}{u_1}\right)$, we see that in the definition of $\sigma_i$ we can replace the ball $B(x_0,r)$ in the Agmon metric by the set $\left\{2\leq X_1^{-1}\left(\frac{u_0}{u_1}\right)\leq 2r\right\}$.  Thus, using the change of variable formula \eqref{ch_variable}, and the fact that $\psi(t)\sim e^{-t}$ when $t\to\infty$, we see that

$$\sigma_0=\lim_{r\to\infty}\sup-\frac{1}{r}\log \int_{2r}^\infty e^{-t} \,\mathrm{d}\chi(t),$$
and

$$\sigma_1= \lim_{r\to\infty}\sup\frac{1}{r}\log \int_1^{2r} e^{t} \,\mathrm{d}\chi(t).$$
But using as above the integration by parts formula for Stieljes measures and the the hypothesis on $\chi$, we see that there is a constant $c>0$ such that for $r>0$ big enough,

$$\int_1^{2r}e^t\, \mathrm{d}\chi(t) \leq c\,e^{(2-\varepsilon)r}$$
and

$$\int_{2r}^\infty e^{-t} \,\mathrm{d}\chi(t)\geq c^{-1}e^{-(2-\varepsilon)r}.$$
This implies at once that $\sigma_i\leq 2$, which concludes the proof.

\end{proof}

\subsection{Volume growth for the improved Hardy inequalities on a convex set}
In this subsection, we show how the general theory developped in subsection 5.2 applies to the particular example of the improved Hardy inequalities on a bounded domain $\Omega$ of $\R^n$. Fix $i\geq0$, and define

$$P:=-\Delta-\mathcal{W}_{i-1}$$
($\mathcal{W}_{-1}=0$ by convention), where $\mathcal{W}_i$ is the weight

$$\mathcal{W}_i:=\frac{1}{4\delta^2}\left(\sum_{k=0}^{i} X_0^2\left(\frac{\delta}{D}\right)\cdots X_k^2\left(\frac{\delta}{D}\right)\right),$$
($\mathcal{W}_{-1}=0$ by convention). Recall also the definition of

$$\mathcal{J}_i:=\mathcal{W}_i-\mathcal{W}_{i-1}=\frac{1}{4\delta^2}X_0^2\left(\frac{\delta}{D}\right)\cdots X_i^2\left(\frac{\delta}{D}\right).$$
From Section 3, recall the definition of $U_{0,j}$ and $U_{1,j}$, and define 

$$u_0:=U_{0,i-1},$$

$$u_1:=U_{1,i-1}.$$
Define as in subsection 5.2, for $k=0,1$

$$\mu_{i,k}:=u_k^2\mathcal{J}_i\mathrm{d}x.$$
Define also the associated volume growth rate $\sigma_{i,k}$, for $k=0,1$. Then as a consequence of Corollary \ref{equality_case}, we have

\begin{Thm}

For every $i\geq0$, the measure $\mu_{i,0}$ (resp. $\mu_{i,1}$) has finite (resp. infinite) mass, and the convergence of volumes is exponential: for $k=0,1$,

$$\frac{\sigma_{i,k}^2}{4}=1.$$

\end{Thm}

\begin{proof}

Denote $\mathrm{d}\chi$ the push-forward measure of $u_0u_1\mathcal{J}_i\mathrm{d}x$ by $X_1^{-1}\left(\frac{u_0}{u_1}\right)$, then the computations done in the proof of Theorem \ref{improved_hardy} show that $\chi$ has linear growth. Also, for $k=0,1$, denote

$$V_k:=\frac{Pu_k}{u_k}.$$
Then by Proposition \ref{Fi-Ter}, in a neighborhood of $\partial \Omega$ we have

$$V_0=\frac{-\Delta\delta}{2\delta}\left(1-R_{i-1}\right),$$
and

$$V_1=\frac{-\Delta\delta}{2\delta}\left(1-R_{i-1}\right)-\frac{-\Delta\delta}{\delta}X_1\left(\frac{\delta}{D}\right)\cdots X_{i+1}\left(\frac{\delta}{D}\right).$$
It is immediate to check that for $k=0,1$,

$$\lim_{\delta\to 0}\frac{V_k}{W}=0.$$
Thus we can apply Corollary \ref{equality_case} to $\mu_{i,k}$, which gives the result.

\end{proof}

\section{Appendix}
Here, we give a proof of the fact that the series

$$\sum_{k\geq1}X_1(t)\cdots X_k(t)$$
converges for every $t\in[0,1)$. We thank A. Tertikas for having provided us with the proof. For every $\varphi$ defined on the unit ball $B_1$, we have the following Hardy inequality (as a simple consequence of Allegretto-Piepenbrink theory, or by direct integration by parts)

$$\int_{B_1} |\nabla u|^2 \; dx \geq \int_{B_1} \frac{ -\Delta \varphi }{ 
\varphi } \; u^2 \; dx, \;\;, u \in C^{\infty}_{0}(B_1)$$
We make the choice

$$\varphi = X^{-1/2}_1(|x|)\cdots X^{-1/2}_k(|x|),$$
and compute (see Lemma 6.3 in \cite{FT})

$$ \begin{array}{rcl}
\frac{ -\Delta \varphi }{ \varphi } &=& \frac{n-2}{2 |x|^2}  \sum_{i=1}^{k} 
X_1(|x|)\cdots X_i(|x|) + \frac{1}{4 |x|^2} \; \sum_{i=1} ^{k} \; X^2_1(|x|)\cdots X^2_i(|x|)\\\\
&\geq &\frac{n-2}{2 |x|^2} \; \sum_{i=1} ^{k} \; X_1(|x|)\cdots X_i(|x|) 
\end{array}$$
Applying it for $n\geq 3$, we conclude the the convergence of the required 
series for $t\in (0,1)$.

\begin{center}{\bf Acknowledgments} \end{center}
The author wishes to thank Y. Pinchover for many interesting discussions and inspiring questions, and A. Tertikas, G. Carron for useful comments. He also thank E. Zuazua for a question about multipolar Hardy inequalities, and the referee for his careful reading of the manuscript. The author acknowledges the support of the Israel Science
Foundation (grant No. 963/11) founded by the Israel Academy of
Sciences and Humanities, and the Technion.

\end{document}